\documentclass[12pt]{article}

\usepackage{amsmath,amsfonts}
\usepackage{caption}
\usepackage{accents}
\usepackage{multirow}

\usepackage{multicol}
\usepackage{amsmath, amsfonts, amssymb}
\usepackage{amsthm}
\usepackage{graphicx}
\graphicspath{ {images/} }
\usepackage{pst-all,pst-infixplot,pst-math,pst-fractal,pst-solides3d}
\usepackage{pict2e}

\newtheorem{teo}{Theorem}[section]
\newtheorem{prop}{Proposition}[section]

\newtheorem{lemma}{Lemma}[section]

\newtheorem{assumption}{Assumption}[section]
\newtheorem{defi}{Definition}[section]

\newcommand{\be}{\begin{eqnarray*}}
\newcommand{\ben}{\begin{eqnarray}}
\newcommand{\ee}{\end{eqnarray*}}
\newcommand{\een}{\end{eqnarray}}
\newcommand{\al}{\begin{align*}}
\newcommand{\eal}{\end{align*}}
\newcommand{\aln}{\begin{align}}
\newcommand{\ealn}{\end{align}}

\let\<\langle
\let\>\rangle
\let\phi\varphi

\begin{document}

\author{Ana Bela Cruzeiro$^*$ ~~~~
Carlos Oliveira$^\dagger$~~~~
Jean Claude Zambrini$^\ddagger$
\\\\$^*$\textit{\small GMFUL and Department of Mathematics of Instituto Superior T\'ecnico}\\\textit{\small Universidade de Lisboa, Av. Rovisco Pais, 1049-001 Lisboa, Portugal}
\\\textit{\small Email: ana.cruzeiro@tecnico.ulisboa.pt}\\
$^\dagger$\textit{\small CEMAPRE, ISEG - School of Economics and Management}\\
	\textit{\small Universidade de Lisboa, Rua do Quelhas 6, 1200-781 Lisboa, Portugal}\\
	\textit{\small GMFUL, Universidade de
	Lisboa, Campo Grande 016, Edif\'i­cio C6, PT-1749-016 Lisboa, Portugal}
\\
\textit{\small	Email: carlosoliveira@iseg.ulisboa.pt}\\
$^\ddagger$\textit{\small GMFUL and Faculdade de Ci\^encias}\\ \textit{\small Universidade de
	Lisboa, Campo Grande 016, Edif\'icio C6, PT-1749-016 Lisboa, Portugal}\\
\textit{\small	Email: jczambrini@fc.ul.pt}
}

\title{ Time-symmetric optimal stochastic control problems in space-time domains }

\maketitle

\begin{abstract}
\noindent
We present a pair of adjoint optimal control problems characterizing a class of time-symmetric stochastic processes defined on random time intervals. The associated PDEs are of free-boundary type.
The particularity of our approach is that it involves two adjoint optimal stopping times adapted to a pair of filtrations, the traditional increasing one and another, decreasing. They are the keys of the time symmetry of the construction, which can be regarded as a generalization of  "Schr\"odinger's problem " (1931-32) to space-time domains. The relation with the notion of "Hidden diffusions" is also described.

\mbox{}
\vspace{2mm}

\mbox{}
\newline
{\bf AMS (2020) Subject Classifications.\/} Primary 93E24;
secondary 60H30; 60H10; 35R35.
\end{abstract}

%================================================

\section{Introduction}
The notion of Bernstein stochastic processes dates back to 1932 (see Bernstein \cite{bernstein1932liaisons}) and  followed from a probabilistic interpretation of a suggestion made by E. Schr\"odinger, one year
before (Schrodinger \cite{schrodinger1932theorie}). During decades this line of ideas attracted very little attention. In 1986 it was shown (Zambrini \cite{zambrini1986variational} and references therein) that, behind it, there is a quantum-like 
regularization method for classical dynamical systems but, in contrast with quantum theory, using well defined probability measures and appropriate path spaces.

More recently, the community of mass transportation theory adopted part of the resulting framework under the logo of ``Schr\"odinger's problem" \cite{leonard2014survey}. It allows, in particular,
to construct very efficient regularizations in numerical approaches to optimal transport problems of interest in imaging, natural  sciences and Economics (Cf., for instance Benamou et al. \cite{benamou2015iterative}, Carlier and Laborde \cite{carlier2018differential}, Di Marino and {Gerolin} \cite{di2019optimal} and Galichon \cite{galichon2018optimal}).

Schr\"odinger's original (one dimensional) problem was to construct random processes interpolating in an optimal way between two ``arbitrary" probability densities, associated
with the heat equation, but given at the boundaries of a fixed time interval $I$. This means in particular that the given future probability had, a priori, nothing to do with the traditional probabilistic interpretation of this parabolic equation.

The answer to this problem, suggested by Schr\"odinger himself, is a class of (``Bernstein") diffusions, generally time inhomogeneous but enjoying a time reversibility property more general than the one known by most probabilists. The probability density of those optimal diffusions has an (integrable) product form of a positive solution of the heat equation and a positive solution of a backward heat equation, both defined on the fixed time interval $I$, respectively with (positive) initial and final boundary conditions.

If we adopt the traditional terminology of Mathematical Physics calling ``Euclidean" any approach of quantum physics where Schr\"odinger's type of equations are replaced by parabolic ones, the probability density of Bernstein diffusions expresses nothing but the Euclidean version of Born's fundamental interpretation of the wave function or, more precisely, of the $L^2$-scalar product of the two wave functions.
After \cite{zambrini1986variational}, the program inspired by Schr\"odinger was developed in various directions, illustrating the generality of its starting idea, in no way limited to the elementary situation considered initially (Zambrini \cite{zambrini2015research}).

The fact that the time interval of Bernstein processes existence was fixed, in Schr\"odinger's problem, is not a necessary or even natural restriction of the method. A natural construction would be to define Bernstein processes in space-time domains. This is the aim of this paper, whose organization is the following.

Section 2 summarizes the original construction of a large class of Bernstein diffusions, on a given (deterministic) time interval. Their particularity is to solve simultaneously two It\^o's
stochastic differential equations, one with an initial boundary condition, the other with a final one. Their relation with the notion of "hidden diffusions" is also indicated.

Section 3 provides a characterization of Bernstein diffusions defined on random time intervals as solutions of two adjoint optimal control problems where pairs of random times and drifts should be optimalized. In terms of partial differential equations these problems are of free-boundary type.
In Section 4 viscosity solutions of the adjoint Hamilton-Jacobi-Bellman underlying our construction are described. Section 5 shows that the solutions of these two boundary value problems are unique. In addition, their relation with Schr\"odinger's original problem and the associated dynamics of Bernstein optimal drifts are given. The characterization of
the distributions of  the two adjoint optimal stopping times used in the construction is the subject of Section 6. It amounts to the construction of a forward and a backward martingale of the process. Section 7 is devoted to a one dimensional example and its discussion.

\section{Bernstein and hidden diffusions stochastic processes}\label{Bernstein stochastic process}
Let $Z\equiv\{Z_t\in\mathbb{R}^n:t\in I\equiv[-T/2,T/2]\}$ be a stochastic process defined on a filtered probability space $\left(\Omega,\Sigma_I,\{{\cal P}_t\}_{t\in I},\{{\cal F}_t\}_{t\in I},P\right)$, where $\{{\cal P}_t\}_{t\in I}$ and $\{{\cal F}_t\}_{t\in I}$ are, respectively, an increasing and a decreasing filtration for the process $Z_t$. 

We say that $Z$ is a Bernstein stochastic process \cite{bernstein1932liaisons} if, for any bounded measurable function $f$, 
\begin{equation}\label{local markov property}
E\left[f(Z_t)\vert {\cal P}_s\cup{\cal F}_u\right]=E\left[f(Z_t)\vert Z_s,Z_u\right],\quad \text{for all } s\leq t\leq u,~[s,u]\subset I.
\end{equation}
This is known today as the local or ``two-sided" Markov property and represents a ``reciprocity" property of the process $Z$ in time. We stress that (1) is weaker than the Markov property. The construction of a Bernstein process relies on the definition of its transition probability $Q\equiv Q(s,x,t,B,u,z)$ that verifies:

 (i) for all $x,z\in\mathbb{R}^n$ and $s< t<u$ in $I$, $B\to Q(s,x,t,B,u,z)$ is a probability measure in the Borel $\sigma-$algebra  ${\cal B}^n$ of $\mathbb{R}^n$;
 
  (ii) for a fixed $B\in {\cal B}^n$ and $s< t<u$ in $I$, $(x,z)\to Q(s,x,t,B,u,z)$ is measurable function;

  (iii) for all $B_1 , B_2 \in  {\cal B}^n$, and $s< t<u<r$ in $I$, 
  
  $$\int_{B_2} Q (s,x,t,B_1 ,u, w)~ Q(s,x,u, dw, r, z)=\int_{B_1} Q(s,x,t,dy,r,z)~ Q(t,y,u,B_2 ,r,z).$$
  \color{black}
  
  Additionally, it can be found in Jamison \cite{jamison1974reciprocal} a proof of the following theorem:
\begin{teo}  
Let $Q$ be a Bernstein transition probability and $m$ a probability measure on ${\cal B}^n\times {\cal B}^n$. Then, there is a unique probability measure
$P=P_m$, such that
\begin{itemize}
\item[(1)] the local Markov property defined in \eqref{local markov property} is satisfied;
\item[(2)] $P_m(Z_{-T/2}\in B_{-T/2},Z_{T/2}\in B_{T/2})=m(B_{-T/2}\times B_{T/2})$, for all $B_{-T/2},B_{T/2}\in {\cal B}^n$;
\item[(3)]$P_m(Z_t\in B\,\vert Z_s,Z_u)=Q(s,Z_s,t,B,u,Z_u)$, for all $-T/2<s\leq t\leq u<T/2$ and $B\in{\cal B}^n$;
\item[(4)] $P_m(Z_{-T/2}\in B_{-T/2},Z_{t_1}\in B_{t_1},\cdots,Z_{t_n}\in B_{t_n},Z_{T/2}\in B_{T/2})$
\begin{align*}
=&\int_{B_{-T/2}\times B_{T/2}}dm(x,z)\int_{B_{t_1}}Q(-T/2,x,t_1,d{y_1},T/2,z)\int_{B_{t_2}}\cdots\\
&\cdots\int_{B_{t_n}}Q(t_{n-1},y_{n-1},t_n,d{y_n},T/2,z).
\end{align*}
\end{itemize} 
\end{teo}

Let $V:\mathbb R^3 \rightarrow \mathbb R$ be a bounded below potential and $\hbar$ be a positive constant such that the integral kernel $h(s,x,t,y)=\left(e^{-\frac{H}{\hbar}(t-s)}\right)(x,y)$, defined on $L^2(\mathbb{R}^n)$, is positive and jointly continuous in $x,y\in\mathbb{R}^n$, where $H$ is a parabolic operator of the form $H=\frac{-\hbar^2}{2} \Delta +V$. Then an appropriate density of Bernstein transition probability can take the form (Cf. \cite{zambrini1986variational}),
$$
Q(s,x,t,dy,u,z)=\frac{h(s,x,t,y)h(t,y,u,z)}{h(s,x,u,z)}.
$$
According to Jamison \cite{jamison1974reciprocal}, there exists a single joint probability measure $m$ that turns $Z$ into a Markov process, given by
$$
m(B_{-T/2}\times B_{T/2})=\int_{B_{-T/2}\times B_{T/2}}\eta^*_{-T/2}(x)h(-T/2,x,T/2,y)\eta_{T/2}(y)dxdy,
$$
where $\eta^*_{-T/2},\eta_{T/2}:\mathbb{R}^n\to\mathbb{R}$ are two arbitrary integrable measurable positive functions. These functions are the unique solutions of a
system of integral equations
$$
\begin{cases}
\eta^*_{-T/2}(x)\int_{\mathbb{R}^n}h(-T/2,x,T/2,z)\eta_{T/2}(z)dz=p_{-T/2}(x)\\
\eta_{T/2}(z)\int_{\mathbb{R}^n}\eta^*_{-T/2}(x)h(-T/2,x,T/2,z)dx=p_{T/2}(z),
\end{cases}
$$ 
for $p_{-T/2}$ and $p_{T/2}$ a given pair of (strictly positive) boundary probability densities. The unique solvability of the above non-linear system was  shown in Beurling \cite{beurling1960automorphism}.  Finally, if $\rho(t,x)$is the density of the process at time $t$,  the probability of the process $Z_t$ being in $B\in{\cal B}^n$ is of the form
$$
P(Z_t\in B)=\int_{B}\rho(t,x)dx =\int_{B} \eta^*(t,x)\eta(t,x)dx
$$
where
\begin{align*}
\eta(t,x)=\int h(t,x,T/2,z)\eta_{T/2}(z)dz\quad , \quad\eta^*(t,x)=\int \eta^*_{-T/2}(y)h^*(-T/2,y,-t,x)dy
\end{align*}
 and $h^*$ is, more generally, the integral kernel of $e^{-(t+T/2 )H^* }$. For $H$ as before $H^* =H$.
One can prove that, in this case, the functions $\eta^*_{t}$ and $\eta_{t}$ are two positive solutions of the initial and terminal problems on $[-\frac{T}{2},\frac{T}{2}]$,
\begin{equation}\label{PDEs}
\begin{cases}
-h\frac{\partial\eta^*}{\partial t}=H\eta^*\\
\eta^*(-T/2,x)=\eta_{-T/2}^*(x)
\end{cases}\quad\text{and}\quad\quad
\begin{cases}
h\frac{\partial\eta}{\partial t}=H\eta\\
\eta(T/2,x)=\eta_{T/2}(x)
\end{cases}.
\end{equation}
This construction was done initially in Zambrini \cite{zambrini1986variational}. For a more rigorous version cf. \cite{cruzeiro1991malliavin}.

Afterwards we are going to focus on this Markovian framework. Let us stress, however, that there are interesting non-Markovian Bernstein processes (Vuillermot and Zambrini \cite{vuillermot2019bernstein}).

Let $B^Z$ and $C^Z$ (resp. $B^Z_*$ and $C^Z_*$) be the forward drift and diffusion coefficient or ``volatility" (resp., backward drift and coefficient) associated with the Bernstein process $Z$ and defined by
\begin{align}
&B^Z(s,x)=\lim_{t\downarrow s}\frac{1}{t-s}\int_{{\mathbb{R}}^3}(y-x)Q(s,x,t,dy,u,z),\\ 
& B^Z_*(u,z)=\lim_{t\uparrow u}\frac{1}{u-t}\int_{{\mathbb{R}}^3}(z-y)Q(s,x,t,dy,u,z)
\end{align}
and 
\begin{align}
&C^Z(s,x)=\lim_{t\downarrow s}\frac{1}{t-s}\int_{{\mathbb{R}}^3}\vert y-x\vert^2 Q(s,x,t,dy,u,z),\\ 
& C^Z_*(u,z)=\lim_{t\uparrow u}\frac{1}{u-t}\int_{{\mathbb{R}}^3}\vert z-y\vert^2 Q(s,x,t,dy,u,z).
\end{align}
For a Markov $3$-d, for instance,  Bernstein process with $H$ defined as before, these functions take the form
\begin{align*}
&B^Z(s,x)=\hbar\nabla\log\eta(s,x),\quad B^Z_*(s,x)=-\hbar\nabla\log\eta^*(s,x)\\
&~~~~~~~~~~~~~~~~~~C^Z(s,x)=C^Z_*(s,x)=\hbar I_{3\times 3},
\end{align*}
where $I_{3\times 3}$ is the identity matrix of dimension 3. For smooth drifts, the Markov Bernstein process solve the forward and backward SDE's:
\begin{align}\label{diffusion Z}
dZ_t&=B^Z(t,Z_t)dt+\hbar^{1/2}dW_t\quad\text{and}\quad
d_*Z_t=B^Z_*(t,Z_t)d_*t+\hbar^{1/2}d_*W^*_t,
\end{align}
{where $W$ represents a Brownian motion adapted to the past filtration and} $W^*$ denotes a Brownian motion adapted to the future filtration. Additionally, it is straightforward to observe that 
\begin{align}\label{relationship between drifts}
B^Z_*(t,x)=B^Z(t,x)-\hbar \nabla\log\rho(t,x).
\end{align}

For sufficiently smooth functions the operators $\cal L$ and $\cal L^*$, defined by
$$
{\cal L}=\partial_t+B^Z\cdot\nabla+\frac \hbar 2\Delta\quad\text{and}\quad 
{\cal L}^*=\partial_t+B_*^Z\cdot\nabla-\frac \hbar 2\Delta,
$$ 
coincide, respectively, with the forward and backward infinitesimal generators of the process $Z$: 
\begin{align*}
&({\cal L}v)(t,x)=\lim_{\bigtriangleup t\downarrow 0}E_{t,x}\left[\frac{v(t+\bigtriangleup t,Z(t+\bigtriangleup t))-v(t,Z(t))}{\bigtriangleup t}\right]\\
&({\cal L}^*v)(t,x)=\lim_{\bigtriangleup t\downarrow 0}E_{t,x}\left[\frac{v(t,Z(t))-v(t-\bigtriangleup t,Z(t-\bigtriangleup t))}{\bigtriangleup t}\right],
\end{align*}
where $E_{t,x}[ \cdot ]$ represents the expected value conditioned on the information that $Z_{t}=x$.

Let us stress that ${\cal L}$ and ${\cal L}^*$ involve indeed the same forward and backward increments as in SDEs \eqref{diffusion Z}.

In a quantum-like context, Bernstein processes are usually seen as critical points of forward and backward action functionals, cf. Cruzeiro and Zambrini \cite{cruzeiro1991malliavin} or Zambrini \cite{zambrini2015research}, among others. In fact, one may see that $\eta(t,x)=e^{-\frac{1}{h}F(t,x)}$ and $\eta^*(t,x)=e^{-\frac{1}{h}F^*(t,x)}$ where $F$ and $F^*$  can be obtained as solutions (or ``value functions") of an optimal control problem. In light of the results derived by Fleming and Soner \cite{fleming2006controlled}, one may state the following result:
\begin{prop}
	Let $F(t,x)$ and $F^*(t,x)$ be classical solutions of Hamilton-Jacobi-Bellman equations 
\begin{align*}
&\begin{cases}
\frac{\partial F}{\partial t}-\frac{1}{2}\vert\nabla F\vert^2+\frac{\hbar}{2}\Delta F+V(x)=0,&\\
F\left(\frac{T}{2},x\right)=F_{T/2}(x),&
\end{cases}t<T/2\text{ and }x\in\mathbb{R}^n\\
&\begin{cases}
\frac{\partial F^*}{\partial t}+\frac{1}{2}\vert\nabla F^*\vert^2-\frac{\hbar}{2}\Delta F^*-V(x)=0\\
F^*\left(-\frac{T}{2},x\right)={F_{-T/2}^*(x)},& 
\end{cases}t>-T/2\text{ and }x\in\mathbb{R}^n.
\end{align*}
Then, if $Z$  solves eqs.\eqref{diffusion Z}, it holds that
\begin{align*}
&F(t,x)= E_{t,x}\left[\int_{t}^{T/2} \left(\frac{1}{2}\vert B^Z(s,Z_s)\vert^2 + V(Z_s) \right)ds+{F_{T/2}(Z_{T/2})}\right]\\
&F^*(t,x)=E_{t,x}\left[\int_{-T/2}^{t} \left(\frac{1}{2}\vert B^Z_*(s,Z_s)\vert^2 + V(Z_s) \right)ds+{F_{-T/2}^*(Z_{-T/2})}\right].
\end{align*}
\end{prop}

Following this approach, originated in Schr\"odinger \cite{schrodinger1932theorie}, the Bernstein process $Z$ is, by construction, well-defined in the domain $[-T/2,T/2]\times\mathbb{R}^n$. But, can we construct a stochastic process satisfying the above time reversibility property, which is not necessarily defined in $[-T/2,T/2]\times\mathbb{R}^n$, but  in a time varying domain contained in $[-T/2,T/2]\times\mathbb{R}^n$? Indeed, one may generalize the concept of Bernstein stochastic processes by using  results derived for ``Hidden diffusions" (Cf. Choy and Nam \cite{choi2004interpolation}). In the context of filtering methods for hidden diffusions, time reversal of the underlying stochastic processes plays an import role as shown in Kim \cite{geun2001filterlng}. In our case the time reversibility is given by the construction of the Bernstein processes itself.

Without taking into account the preceding construction, let us only assume that $Z$ satisfies the forward and backward SDEs \eqref{diffusion Z} where $B^Z(t,x)$ is bounded and uniformly Lipschitz continuous for $(t,x)\in [-T/2, T/2]\times\mathbb{R}^n$, $B^Z_*(t,x)$ satisfies \eqref{relationship between drifts} and $\rho(t,x)$ is the solution of the {forward} Kolmogorov equation for the process $Z_t$. An auxiliary process $Y_t$ taking values in $\mathbb{R}^n\cup\text{"hidden"}$  is defined as
$$
Y_t=\begin{cases}
Z_t,& \text{if } Z_t\notin A\\
"hidden",& \text{if } Z_t\in A
\end{cases}
$$
where $A$ is a Borel set in $[-T/2, T/2]\times\mathbb{R}^n$. If $Y_{[-T/2,T/2]}$ represents the $\sigma-$algebra $\sigma\{Y_t:\,-T/2\leq t\leq T/2\}$, $\tau_A=\inf\{u>t:\,(u,Z_u)\notin A\}$ and $\tau^*_A=\sup\{s<t:\,(s,Z_s)\notin A\}$\footnote{$\tau_A$ is a ${\cal P}_t$-stopping time and $\tau_A^*$ is a ${\cal F}_t$-stopping time.} then it is straightforward to see that
$$
E\left[f(Z_t)\,\vert Y_{[-T/2,T/2]}\right]=E_{t,x}\left[f(Z_t)\,\vert {\cal P}_{\tau_A}\cup {\cal F}_{\tau^*_A}\right].
$$
Combining the decomposition 
\begin{align*}
E\left[Z_t\,\vert {\cal P}_{\tau_A}\cup {\cal F}_{\tau^*_A}\right]&=E\left[Z_t{\bf 1}_{Z_t\notin A}\,\vert {\cal P}_{\tau_A}\cup {\cal F}_{\tau^*_A}\right]+E\left[Z_t{\bf 1}_{Z_t\in A}\,\vert {\cal P}_{\tau_A}\cup {\cal F}_{\tau^*_A}\right]
\end{align*}
with the strong Markov property and Theorem 4 in Choy and Nam \cite{choi2004interpolation}, one may state the following result:
\begin{teo}
	For any bounded Borel function $f$, we have
\begin{align*}
E\left[f(Z_t)\,\vert Y_{[-T/2,T/2]}\right]=E\left[f(Z_t)\,\vert {\cal P}_{\tau_A}\cup {\cal F}_{\tau^*_A}\right]=E[f(Z_t)\,\vert \tau_A,Z_{\tau_A}, \tau_A^*,Z_{\tau_A^*}].
\end{align*}
\end{teo}
This property generalizes indeed  the local Markov property \eqref{local markov property} since one may consider the particular domain
$$
A=\left((s,u)\times\mathbb{R}^n\right)\cup(]-\infty,s]\cup[u,+\infty[\times\{\infty\}).
$$
For this particular case, it is straightforward to recover  equality (1):
$$E\left[f(Z_t)\,\vert {\cal P}_{s}\cup {\cal F}_{u}\right]=E[f(Z_t)\,\vert Z_{s}, Z_{u}].$$

Let $B$ be a Borel set contained in $A$ and $\tilde{Q}$ be the transition probability of $Z$ on the event $\{Z_t\in A\}$ that is defined as
 $$\tilde Q(\tau_A,Z_{\tau_A},t,B,\tau^*_A,Z_{\tau^*_A})=P(Z_t\in B\,\vert \tau_A,Z_{\tau_A}, \tau_A^*,Z_{\tau_A^*}).$$ 
According to Choy and Nam \cite{choi2004interpolation}, the conditional density is given by
$$
\tilde Q(s,x,t,dy,u,z)=\frac{\tilde h(s,x,t,y,u,z)}{\int_A \tilde h(s,x,t,y,u,z)dy},
$$
where $\tilde h(s,x,t,y,u,z)$ represents the joint density of $(\tau_A,Z_{\tau_A},Z_t,\tau^*_A,Z_{\tau^*_A})$. Additionally, $\tilde h(s,x,t,y,u,z)$ admits the following decomposition 
$$
\tilde h(s,x,t,y,u,z)=L_\tau(s,x,t,y)\rho(t,y)L_{\tau^*}(t,y,u,z),
$$
where $L_{\tau_A}$ (resp., $L_{\tau^*_A}$) is the joint density of $(\tau_A,Z_{\tau_A}, Z_t)$ (resp., $(\tau^*_A,Z_{\tau^*_A}, Z_t)$). {Furthermore, in light of Lemma 2, in Choy and Nam \cite{choi2004interpolation}, and Lemma 3.2, in Kim and Nam \cite{kim2001optimal}, the density functions} $\Psi_{s,x}\equiv\Psi_{s,x}(t,y)=L_{\tau_A}(s,x,t,y)$ and $\Psi^*(t,y)\equiv\Psi^*_{u,z}(t,y)=L_{\tau^*_A}(t,y,u,z)$) are the solutions to the boundary problems:
\begin{align*}
\begin{cases}
\frac{\partial\Psi_{s,x}}{\partial t}+B^Z\cdot\nabla\Psi_{s,x}+\frac{\hbar}{2}\Delta\Psi_{s,x}=0,&(t,y)\in A\\
\Psi(t,y)=\delta_{(s,x)}(t,y),&(t,y)\in \partial A\\
\Psi(T/2,y)=0,&(T/2,y)\in A
\end{cases}
\end{align*}
\text{and}
\begin{align*}
\begin{cases}
\frac{\partial\Psi^*}{\partial t}+B^Z_*\cdot\nabla\Psi^*-\frac{\hbar}{2}\Delta\Psi^*=0,&(t,y)\in A\\
\Psi^*(t,y)=\delta_{(u,z)}(t,y),&(t,y)\in \partial A\\
\Psi^*(-T/2,y)=0,&(-T/2,y)\in A
\end{cases}.
\end{align*}

%Through this paper, we will work with functionals of the Brownian motion, which requires for instance an appropriate integration by parts formula. To present this, we introduce the space $\Omega_0=\{\omega\in C([0,T/2];\mathbb{R}^n):\omega(0)=0\}$ and the Cameron-Martin space,
%$$
%{\cal H}=\left\{\phi\in\Omega_0:\phi\text{ exists and}\int_0^{T/2}\vert\phi\vert^2ds<+\infty\right\},
%$$   
%which is, in fact, an Hilbert space when we add the inner product $(\cdot\vert\cdot)$, defined by $(\phi_1\vert\phi_2)=\int_{0}^{T/2}\phi_1(s)\vert\phi_2(s)ds$.
%
%
%\begin{equation}\label{diffusion}
%dZ_t=b(Z_t)dW_t+h^{1/2}dW_t
%\end{equation}

\section{Stochastic optimal control problems}
{In this section, we will argue that Bernstein stochastic processes may be introduced in light of stochastic control and optimal stopping theories. 	To do this, one has to introduce the corresponding action functionals and the control diffusions.} 

Let $\cal U$ (resp., ${\cal U}^*$) be the set of functions $b$ (resp. $b^*$) such that the processes $b(s,Z_s)$ (resp. $b^*(s,Z_s)$) are progressively measurable processes valued in a compact metric separable %Let $\cal U$ (resp., ${\cal U}^*$) be the set of all progressively measurable processes valued in a compact metric separable 
space $M$ (resp., ${ M}^*$), with respect to the increasing filtration $\{{\cal P}_t\}_{t\in I}$ (resp., decreasing $\{{\cal F}_t\}_{t\in I}$) and ${\cal T}_{t}$ (resp., ${\cal T}^*_{t}$\color{black}) be the set of all stopping times adapted to the  filtration $\{{\cal P}_t\}_{t\in I}$ (resp.,$\{{\cal F}_t\}_{t\in I}$) that are greater (resp., less) then or equal to $t$.

Consider the action functionals $J_{t,x}$ and $J^*_{t,x}$ defined by
\begin{align}\label{forward J}
	J_{t,x}(Z;\tau,b)&=E_{t,x}\left[\int_{t}^{\tau\wedge T/2} \left(\frac{1}{2}\vert b(s,Z_s)\vert^2 + V(Z_s) \right)ds+S(Z_{\tau\wedge T/2})\right],\\
	dZ_u&=b(u,X_u)du+\hbar^{1/2}dW_u ,\quad Z_t=x\text{ and }-\frac{T}{2}\leq t\leq u\leq \frac{T}{2}\label{forward process}
\end{align}
and 
\begin{align}\label{backward J}
	J^*_{t,x}(Z;\tau^*,b^*)&=E_{t,x}\left[\int_{-T/2\vee\tau^*}^{t} \left(\frac{1}{2}\vert b^*(s,Z_s)\vert^2 + V(Z_s) \right)ds+S^*(Z_{-T/2\vee\tau^*})\right],\\
 d_*Z_s&=b^*(s,Z_s)d_*s+\hbar^{1/2}d_*W^*_s ,\quad Z_t=x\text{ and }-\frac{T}{2}\leq s\leq t\leq \frac{T}{2}\label{backward process}
\end{align}
where $\tau \in {\cal T}_{t}$, $\tau^*  \in {\cal T}^*_{t}$, $S$ is the terminal  {boundary} condition in the functional $J_{t,x}$ and $S^*$ is the initial  {boundary} condition in the functional $J^*_{t,x}$.  {In this context, $d_*$ should be understood as the backward differential used above in ${\cal L}^*$.} $J$ is usually called the forward action function with terminal condition and $J^*$ the backward action functional with initial condition. The stochastic optimal control problems consist in finding $\hat\tau$ and $\hat b$ (resp., $\hat\tau^*$ and $\hat b^*$) that minimize the function $J$ (resp., $J^*$). Equivalently, we can look for  the ``value functions" ${U}$ and $U^*$ given by
\begin{align}
&U(t,x)=\inf_{(b,\tau)\in{\cal U}\times{\cal T}_{t}} J_{t,x}(Z;\tau,b)\,{=J_{t,x}(Z;\hat\tau,\hat b)}\label{optimal U}\\
&U^*(t,x)=\inf_{(b^*,\tau^*)\in{\cal U}^*\times{\cal T}^*_{t}} J^*_{t,x}(Z;\tau^*,b^*)\,{=J^*_{t,x}(Z;\hat \tau^*,\hat b^*)}\label{optimal U*}.
\end{align}

Throughout the paper, we will refer to the set $${\cal C}=\{(t,x)\in[-T/2,T/2)\times\mathbb{R}^n:U(t,x)<S(x)\}$$
as  forward continuation region and as forward stopping region, the set $${\cal S}=\{(t,x)\in[-T/2,T/2]\times\mathbb{R}^n:(t,x)\notin {\cal C}\}.$$ Additionally, the backward continuation and stopping regions will be denoted, respectively, by \begin{align*}
&{\cal C}^*=\{(t,x)\in(-T/2,T/2]\times\mathbb{R}^n:U^*(t,x)<S^*(x)\} \text{ and } \\
&{\cal S}^*=\{(t,x)\in[-T/2,T/2]\times\mathbb{R}^n:(t,x)\notin {\cal C}^*\}.
\end{align*}
Henceforward, we assume that the functions $V, S$ and $S^*$ satisfy the next assumption.
\begin{assumption}\label{regularity of V and S}
	Let the potential $V$, the terminal condition $S$ and the initial condition $S^*$ be such that
	\begin{itemize}
\item[1)] $V$, $S$ and $S^*$ are Lipschitz continuous;
\item[2)] $V$ is a lower-bounded function and $S$ and $S^*$ are such that
 $\{S(X_\tau)\}_{\tau\in {\cal T}_{t}}$\text{ and }$\{S^*(X_\tau)\}_{\tau \in {\cal T}^*_{t}}$
\color{black} \text{ are two uniformly integrable families of random variables.}	\end{itemize}

\end{assumption}
Taking into account  Assumption \ref{regularity of V and S}, one can easily prove that the functions $(x,t)\to U(x,t)$ and $(x,t)\to U^*(x,t)$ are continuous. A proof of the next result, for the forward version $U$, can be found in Pham \cite{pham1998optimal}. Additionally, the continuity of $U^*$ can be obtained by using the same type of arguments.
\begin{prop}
	Consider the value functions $U$ and $U^*$ defined as in \eqref{optimal U} and \eqref{optimal U*}. Then, $U^*\in C^0([-T/2,T/2]\times\mathbb{R}^n)$ and $U\in C^0([-T/2,T/2]\times\mathbb{R}^n)$. Additionally, $x\to U(t,x)$ and $x\to U^*(t,x)$ are Lipschitz continuous, uniformly in $t$.
\end{prop}  
In the next section, we will need to notice that, in light of 2) in Assumption \ref{regularity of V and S},
 $U$ and $U^*$ satisfy the following property

\begin{align}\label{uniformly integrability}
&\{U(\tau,Z_\tau)\}_{\tau\in {\cal T}_{t}}\text{ and }\{U^*(\tau^*,Z_{\tau^*})\}_{\tau^*\in {\cal T}^*_{t}}\text{ are two uniformly integrable}\\\nonumber
&~~~~~~~~~~~~~~~~~~~~~~~ \text{ families of random variables.}
\end{align}
To prove this, one should observe that for a single random variable uniformly integrability means that its expected value is finite. Additionally, taking 
into account that
\begin{align*}
-\infty<U(t,x)\leq S(x)\quad\text{and}\quad -\infty<U^*(t,x)\leq S^*(x),
\end{align*}
the first inequality following from Assumption \ref{regularity of V and S}, we deduce  that
\begin{align*}
&\int_{t}^{\hat{\tau}} \frac{1}{2}\vert b(s,Z_s)\vert^2 + V(Z_s) ds+S(Z_{\hat{\tau}})\\
&\int_{\hat{\tau}^*}^{t} \frac{1}{2}\vert b^*(s,Z_s)\vert^2 + V(Z_s) ds+S^*(Z_{\hat{\tau}^*})
\end{align*}
are two uniformly integrable random variables. In addition, this is equivalent to say that there is a uniformly integrable test function $f:[0,\infty)\to[0,\infty)$, (see Definition C.2 and Theorem C.3 in {\O}ksendal \cite{oksendal2003stochastic}) 
such that  
\begin{align*}
&E_{t,x}\left[f\left(\left\vert\int_{t}^{\hat{\tau}} \frac{1}{2}\vert b(s,Z_s)\vert^2 + V(Z_s) ds+S(Z_{\hat{\tau}})\right\vert\right)\right]<\infty\\
&E_{t,x}\left[f\left(\left\vert\int_{\hat{\tau}^*}^{t} \frac{1}{2}\vert b^*(s,Z_s)\vert^2 + V(Z_s) ds+S^*(Z_{\hat{\tau}^*})\right\vert\right)\right]<\infty
\end{align*}
\begin{prop} 
Let $U$ and $U^*$ be defined as in equation \eqref{optimal U} and \eqref{optimal U*}. Then $U$ and $U^*$ {satisfy property \eqref{uniformly integrability}.}
\end{prop}
\begin{proof}
	Pick a uniform integrability test function $f:[0,\infty)\to[0,\infty)$, which is increasing and convex, and notice that
	\begin{align}
	\nonumber E_{t,x}&\left[f\left(\left\vert U(\tau,Z_\tau)\right\vert\right)\right]\leq E_{t,x}\left[f\left(\left\vert\int_{t}^{\hat{\tau}} \frac{1}{2}\vert b(s,Z_s)\vert^2 + V(Z_s) ds+S(Z_{\hat{\tau}})\right\vert\right)\right]<\infty,
	\end{align}
	for any $\tau\in{\cal T}_t$, where the first inequality follows from the strong Markov property and  Jensen's inequality, while the second inequality follows from the comments above. 
\end{proof}

In this paper, our main goal is to show that  stochastic processes obtained as solutions of such stochastic control problems will be  Bernstein diffusions. To do this, in the next sections we will show that the value functions $U$ and $U^*$ are viscosity solutions for the correspondent HJB equations.

\section{Dynamic programming principle and viscosity solutions}\label{Dynamic programming principle and viscosity solutions}
The relationship between stochastic control problems and PDE's is well known: usually the value function can be recovered as a solution to a suitable HJB equation (see, for instance Fleming and Soner \cite{fleming2006controlled}). As, often, the value function associated with the control problem is not $C^2$,  we will prove in this section that the value functions $U$ and $U^*$ are, respectively, viscosity solutions to the HJB equations
\begin{align*}
& %\max\left\{\max_{b\in{\cal U}}\left\{-{\cal L} v(t,x)-\frac{h}{2}\Delta v-\frac{1}{2}\vert b(t,x)\vert^2-V(x) \right\},v-S(x)\right\}
\max\{{\cal H}^b(t,x,v,\partial_t v,\nabla v,\Delta v),v-S(x)\}=0, \text{ in } [-T/2,T/2)\times\mathbb{R}^n\\
& %\max\left\{\max_{b\in{\cal U^*}}\left\{{\cal L}^*v^*-\frac{h}{2}\Delta v^*-\frac{1}{2}\vert b^*(t,x)\vert^2-V(x)\right\},v^*-S^*(x)\right\}
\max\{{\cal H}_*^{b^*}(t,x,v^*,\partial_t v^*,\nabla v^*,\Delta v^*),v^*-S^*(x)\}=0, \text{ in } (-T/2,T/2]\times\mathbb{R}^n,
\end{align*}
where   {``generalized Hamiltonians''} can be defined here by
\begin{align*}
{\cal H}^b(t,x,v,\partial_t,\nabla v,\Delta v)&=\max_{b\in{\cal U}}\left\{-\frac{\partial v}{\partial t}-b\cdot\nabla v-\frac{\hbar}{2}\Delta v-\frac{1}{2}\vert b\vert^2-V(x) \right\}\\
&=-\frac{\partial v}{\partial t}+\frac{1}{2}\vert \nabla v\vert^2 -\frac{\hbar}{2}\Delta v-V(x)
\\
{\cal H}^{b^*}_*(t,x,v^*,\partial_t,\nabla v^*,\Delta v^*)&=\max_{b\in{\cal U^*}}\left\{\frac{\partial v^*}{\partial t}\textcolor{red}{+}b^*\cdot\nabla v^*-\frac{\hbar}{2}\Delta v^*-\frac{1}{2}\vert b^*\vert^2-V(x)\right\}\\
&=\frac{\partial v^*}{\partial t}+\frac{1}{2}\vert\nabla v^*\vert^2-\frac{\hbar}{2}\Delta v^*-V(x)
\end{align*}
In this case, the minimizing controls would be given by
\begin{equation}\label{optimal b}
	\hat b(t,x)=-\nabla U(t,x)\quad\text{and}\quad \hat b^*(t,x)=\nabla U^*(t,x),
\end{equation}
and the HJB equations reduce to
\begin{align}\label{HJB v}
&\max\left\{-\frac{\partial v}{\partial t}+\frac{1}{2}\vert\nabla v\vert^2-\frac{\hbar}{2}\Delta v-V(x),v-S(x)\right\}=0\\\label{HJB v*}
&\max\left\{\frac{\partial v^*}{\partial t}+\frac{1}{2}\vert\nabla v^*\vert^2-\frac{\hbar}{2}\Delta v^*-V(x),v^*-S^* (x)\right\}=0.
\end{align}
Additionally, $v$ and $v^*$ satisfy  terminal and initial conditions
\begin{equation}\label{boundary condition}
v\left(\frac{T}{2},x\right)=S(x)\quad\text{and}\quad v^*\left(-\frac{T}{2},x\right)=S^*(x), \forall x\in\mathbb{R}^n.
\end{equation} 
\begin{defi}\label{Definition Viscosity Solution}
	Consider a locally bounded function $v:[-T/2,T/2)\times\mathbb{R}^n\to\mathbb{R}$. Then, $v$ is a
	\begin{itemize}
		\item[(a)] viscosity subsolution to \eqref{HJB v} if whenever $\psi\in C^2([-T/2,T/2)\times\mathbb{R}^n)$ and ${v}-\psi$ has a local maximum at $(t,x)\in [-T/2,T/2)\times\mathbb{R}^n$, such that ${v}(t,x)=\psi(t,x)$, then
		\begin{align*}
		&\max\left\{-\frac{\partial \psi}{\partial t}+\frac{1}{2}\vert\nabla \psi\vert^2-\frac{\hbar}{2}\Delta \psi-V(x),v-S(x)\right\}\leq 0;
		\end{align*}
		\item[(b)] viscosity supersolution to \eqref{HJB v} if whenever $\psi\in C^2([-T/2,T/2)\times\mathbb{R}^n)$ and ${v}-\psi$ has a local minimum at $(t,x)\in [-T/2,T/2)\times\mathbb{R}^n$, such that ${v}(t,x)=\psi(t,x)$, then
		\begin{align*}
		&\max\left\{-\frac{\partial \psi}{\partial t}+\frac{1}{2}\vert\nabla \psi\vert^2-\frac{\hbar}{2}\Delta \psi-V(x),v-S(x)\right\}\geq 0;
		\end{align*}
		\item[(c)] viscosity solution to \eqref{HJB v} if it is simultaneously a viscosity subsolution and a viscosity supersolution to \eqref{HJB v}. 
	\end{itemize}
A viscosity solution for the HJB equation \eqref{HJB v*} can be defined in the same way.
\end{defi} 
To reach the main result of this section, one needs to state a suitable Bellman principle for the control problems \eqref{optimal U} and \eqref{optimal U*}. For further details about Bellman's principle for these control problems, we can refer to Krylov \cite{krylov2008controlled} (see also Pham \cite{pham1998optimal}). 

Fixing $\epsilon>0$, $b\in{\cal U}$ and $b^*\in{\cal U}^*$,  one can define the following two stopping times:
\begin{align}
\tau_{t,x,b,\epsilon}=\inf\{t\leq s\leq T/2\,:U(s,Z_s)\geq S(Z_s)-\epsilon\}\quad(\hbox{a}~ {\cal P}_t-\text{stopping time})\\
{\tau^*}_{t,x,b^*,\epsilon}=\sup\{t\leq s\leq T/2\,:U^*(s,Z_s)\geq S^*(Z_s)-\epsilon\}\quad(\hbox{a} ~ {\cal F}_t-\text{stopping time}).
\end{align}
\begin{prop}
Let $(t,x)\in[-T/2,T/2]\times\mathbb{R}^n$ and $\epsilon>0$. Then, if $\tau_b\leq\tau_{t,x,b,\epsilon}$, for all $b\in{\cal U}$,
\begin{equation*}
U(t,x)=\inf_{b\in{\cal U}}E_{t,x}\left[\int_t^{\tau_b} \frac{1}{2}\vert b(s,Z_s)\vert^2+V(Z_s) ds+U(\tau_b,Z_{\tau_b})\right].
\end{equation*}
Similarly, if $(t,x)\in[-T/2,T/2]\times\mathbb{R}^n$, $\epsilon>0$ and $\tau^*_{b^*}\geq\tau^*_{t,x,b^*,\epsilon}$, then
\begin{align*}
U^*(t,x)=\inf_{b^*\in{\cal U}^*}E_{t,x}\Bigg[\int_{\tau^*_{b^*}}^t  \frac{1}{2}\vert b^*(s,Z_s)\vert^2+V(Z_s)ds+U^*(\tau^*_{b^*},Z_{\tau^*_{b^*}})\Bigg].
\end{align*}
\end{prop}
It is now possible to state the existence  of solution for the adjoint boundary problems defined above.
\begin{prop}\label{existence of viscosity solution}
Consider the forward and backward stochastic optimal control problems defined respectively by \eqref{forward J}-\eqref{forward process}-\eqref{optimal U} and \eqref{backward J}-\eqref{backward process}-\eqref{optimal U*} and Assumption \ref{regularity of V and S}. Then $U$ and $U^*$ are, respectively, viscosity solutions to equations \eqref{HJB v} and \eqref{HJB v*}. Additionally, the following conditions are satisfied
$$
U\left(\frac{T}{2},x\right)=S(x)\quad\text{and}\quad U^*\left(-\frac{T}{2},x\right)=S^*(x), \forall x\in\mathbb{R}^n.
$$ 
\end{prop}
\begin{proof}The proof of the result for the forward and backward cases is similar. Therefore, we will focus our attention on the backward case which is less common.

To prove that $U^*$ is a viscosity solution to the HJB equation \eqref{HJB v*} and satisfies the boundary condition $U^*(-T/2,x)=S^*(x)$, we will split the proof in three steps.
\vskip 3mm
\textbf{(i) Supersolution property:}

\vskip 2mm

 Let $(t,x)\in[-T/2,T/2)\times\mathbb{R}$ and $\psi\in C^2([-T/2,T/2)\times\mathbb{R})$ be such that $(x,t)$ is a local minimizer of $U^*-\psi$ and $U^*(t,x)=\psi(t,x)$. We start by noticing that for every $(t,x)\in{\cal S}$, the forward stopping region of Section 3, we have $U^*(t,x)=S^*(x)$. Fix $(t,x)\in{\cal C}$, the forward continuation region,  and let $\theta^*_{b^*}\in (\tau^*_{t,x, b^*,\epsilon},t)$ be such that $Z_{s}$ starts at $x$ and stays in a neighborhood $N(x)$ for $\theta^*_{b^*}\leq s\leq t$. Therefore, from the dynamical programming principle, we have
\begin{align}\nonumber
U^*(t,x)&= {\inf_{b^*\in {\cal U^*}}E_{t,x}\Bigg[\int_{\theta^*_{b^*}}^t\frac{1}{2}\vert b^*(s,Z_s)\vert^2+V(Z_s)ds+U^*(\theta_u,Z_{\theta_u})\Bigg]}\\
&\geq \inf_{b^*\in {\cal U}^*} E_{t,x}\Bigg[\int_{u}^t\frac{1}{2}\vert  b^*(s,Z_s)\vert^2+V(Z_s)ds+\psi^*(u,Z_{u})\Bigg].\label{Eq SupP 1}
\end{align}
 Applying  Dynkin's formula in $E_{t,x}\Bigg[\psi(\theta^*_{b^*},Z_{\theta^*_{b^*}})\Bigg]$, we get 
 \begin{align}
 U^*(t,x)-E_{t,x}\Bigg[\psi(\theta^*_{b^*},Z_{\theta^*_{b^*}})\Bigg]=E_{t,x}\Bigg[\int_{\theta^*_{b^*}}^{t}\frac{\partial \psi}{\partial t}(s,Z_s)+ b^*(s,Z_s)\cdot\nabla \psi-\frac{\hbar}{2}\Delta \psi(s,Z_s)ds\Bigg].\label{Eq SupP 2}
 \end{align}
  Consequently, combining \eqref{Eq SupP 1} and \eqref{Eq SupP 2}, it follows that
 \begin{align*}
 0&\leq {\inf_{b^*\in {\cal U^*}}}E_{t,x}\Bigg[\int_{\theta_u}^t\frac{\partial \psi}{\partial t}(s,Z_s)+ b^*(s,Z_s)\cdot\nabla \psi-\frac{\hbar}{2}\Delta \psi(s,Z_s)-\frac{1}{2}\vert b^*(s,Z_s)\vert^2-V(Z_s)ds\Bigg].
% &=\frac{1}{u}E_{t,x}\Bigg[\int_{\theta_u}^t\frac{\partial v^*}{\partial t}(s,Z_s)+\frac{1}{2}\vert\nabla v^*(s,Z_s)\vert^2-\frac{\hbar}{2}\Delta v^*(s,Z_s)-V(x)ds\Bigg].
 \end{align*}
 By letting  $\theta^*_{b^*}\nearrow t$ and dividing by $E[\theta^*_{b^*}]$, we get
  \begin{align*}
 0&\leq {\inf_{b^*\in {\cal U^*}}}\left\{\frac{\partial \psi}{\partial t}(t,x)+ b^*(t,x)\cdot\nabla \psi-\frac{\hbar}{2}\Delta \psi(t,x)-\frac{1}{2}\vert b^*(t,x)\vert^2-V(x) \right\}
 \end{align*}
 and, consequently,
 \begin{equation*}
  \frac{\partial \psi^*}{\partial t}+\frac{1}{2}\vert\nabla \psi^*\vert^2-\frac{\hbar}{2}\Delta \psi^*-V(x)\geq 0\quad\text{for all } (t,x)\in(-T/2,T/2]\times\mathbb{R}^n.
\end{equation*}
\vskip 3mm

\textbf{(ii) Subsolution property:}

\vskip 2mm
 Let $(t,x)\in[-T/2,T/2)\times\mathbb{R}$ and $\psi\in C^2([-T/2,T/2)\times\mathbb{R})$ be such that $(x,t)$ is a local maximizer of $U^*-\psi$ and $U^*(x,t)=\psi(x,t)$. From the dynamical programming principle, we have that, for any $s\leq t$
 \begin{align}\nonumber
 U^*(t,x)&\leq E_{t,x}\Bigg[\int_{s}^t\frac{1}{2}\vert  b^*(u,Z_u )\vert^2+V(Z_u )du+U^*(s,Z_{s})\Bigg]\\
 &\leq E_{t,x}\Bigg[\int_{s}^t\frac{1}{2}\vert  b^*(u,Z_u )\vert^2+V(Z_u )du+\psi(s,Z_{s})\Bigg].\label{Eq SubP 1}
 \end{align}
 Consequently, combining \eqref{Eq SupP 2} and \eqref{Eq SubP 1} and using a similar argument to the one used in the proof of the supersolution property, it follows that
  \begin{align*}
0&\geq \frac{1}{s}E_{t,x}\Bigg[\int_{s}^t\frac{\partial \psi}{\partial t}(u,Z_u )+ b^*(u,Z_u )\cdot\nabla \psi-\frac{\hbar}{2}\Delta \psi(u,Z_u )-\frac{1}{2}\vert b^*(u,Z_u )\vert^2-V(Z_u )du\Bigg].
 \end{align*}
 Therefore, letting $s\nearrow t$ and using the dominated convergence theorem, we get that
 \begin{equation}
 0\geq 	\frac{\partial \psi}{\partial t}(t,x)+ b^*(t,x)\cdot\nabla \psi-\frac{\hbar}{2}\Delta \psi(t,x)-\frac{1}{2}\vert b^*(t,x)\vert^2-V(x).
 \end{equation}
 Since $b^*$ is an arbitrary control, we have the required result:
 \begin{equation}\label{Eq SubP 3}
 0\geq  \frac{\partial \psi^*}{\partial t}+\frac{1}{2}\vert\nabla \psi^*\vert^2-\frac{\hbar}{2}\Delta \psi^*-V(x).
 \end{equation}
 Combining \eqref{Eq SubP 3} with the fact that $U^*(t,x)\leq J^*_{t,x}(Z;t,b)=S^*(x)$, this naturally implies that
$$ \max\left\{\frac{\partial \psi^*}{\partial t}+\frac{1}{2}\vert\nabla \psi^*\vert^2-\frac{\hbar}{2}\Delta \psi^*-V(x),\psi^*-S^*(x)\right\}\leq 0.$$
\vskip 3mm
 \textbf{(iii) Boundary condition:} 
 
 \vskip 2mm
 
 By construction, $J^*_{-T/2,x}(Z,\tau^*,b)=S^*(x)$ for all $\tau^* \in {\cal T}_{-T/2}$ (a ${\cal P}_{-\frac{T}{2}}$ - stopping time). Therefore, one can trivially conclude that $U^*(-T/2,x)=S^*(x)$.
\end{proof}
\vskip 3mm

 Until the next section we will state some auxiliary results that will be useful to prove an uniqueness result. To present these results we let $\tilde S:[-T/2,T/2]\times\mathbb{R}^n\to\mathbb{R}$ and $\tilde S^*:[-T/2,T/2]\times\mathbb{R}^n\to\mathbb{R}$ be two continuous functions, $A\subset [-T/2,T/2]\times\mathbb{R}^n$ be an open bounded set, $\tau_A=\inf\{u>t:\,(u,Z_u )\notin A\}$ and $\tau^*_A=\sup\{s<t:\,(s,Z_s)\notin A\}$.
\begin{lemma}\label{remark}
	Consider the two modified optimal stopping problems 
	\begin{align*}
	\tilde U(x,t)&=\inf_{(b,\tau)\in{\cal U}\times{\cal T}_{t}}E_{t,x}\left[\int_t^{\tau\wedge \tau_A}\left( \frac{1}{2}\vert b(s,Z_s)\vert^2+V(Z_s) \right)ds+\tilde{S}(\tau\wedge \tau_A,Z_{\tau\wedge \tau_A})\right],\\
	\tilde U^*(x,t)&=\inf_{(b^*,\tau^*)\in{\cal U}^*\times{\cal T}^*_{t}}E_{t,x}\left[\int_{\tau_A^*\vee\tau^*}^{t}\left( \frac{1}{2}\vert b^*(s,Z_s)\vert^2 + V(Z_s) \right) ds+\tilde S^*(\tau_A^*\vee\delta,Z_{\tau_A^*\vee\delta})\right].
	\end{align*}
 Then, the value functions $\tilde U:A\to\mathbb{R}$ and $\tilde U^*:A\to\mathbb{R}$  are respectively viscosity solutions of the adjoint boundary problems
	\begin{align}\label{Auxiliary boundary problem forward}
	\begin{cases}
	&\max\left\{-\frac{\partial v}{\partial t}+\frac{1}{2}\vert\nabla v\vert^2-\frac{\hbar}{2}\Delta v-V(x),v-\tilde S(t,x)\right\}=0,\quad (t,x)\in  A\\
	&v(t,x)= \tilde S(t,x),\quad (t,x)\in \partial A
	\end{cases},
	\end{align}
	and
	\begin{align}\label{Auxiliary boundary problem backward}
	\begin{cases}
	&\max\left\{\frac{\partial v^*}{\partial t}+\frac{1}{2}\vert\nabla v^*\vert^2-\frac{\hbar}{2}\Delta v^*-V(x),v^*-\tilde S^*(t,x)\right\}=0,\quad (t,x)\in  A\\
	&v^*(t,x)= \tilde S^*(t,x), \quad (t,x)\in \partial A
	\end{cases}.
	\end{align}
\end{lemma}

This result can be proven by similar arguments to the ones  of Proposition \ref{existence of viscosity solution}. In the next lemma, we prove that solutions $v$ and $v^*$ of   \eqref{Auxiliary boundary problem forward} and \eqref{Auxiliary boundary problem backward} are also solutions to the boundary problems
	\begin{align}\label{Auxiliary boundary problem forward 2}
\begin{cases}
&-\frac{\partial v}{\partial t}+\frac{1}{2}\vert\nabla v\vert^2-\frac{\hbar}{2}\Delta v-V(x)=0,\quad (t,x)\in  A_v\\
&v(t,x)= \tilde S(t,x),\quad (t,x)\in \partial A_v
\end{cases},
\end{align}
and
\begin{align}\label{Auxiliary boundary problem backward 2}
\begin{cases}
&\frac{\partial v^*}{\partial t}+\frac{1}{2}\vert\nabla v^*\vert^2-\frac{\hbar}{2}\Delta v^*-V(x)=0,\quad (t,x)\in  A^*_{v^*}\\
&v^*(t,x)= \tilde S^*(t,x),\quad (t,x)\in \partial A^*_{v^*}
\end{cases},
\end{align}
where 
$$
A_v=\left\{(t,x):\,v(t,x)< \tilde{S}(t,x)\right\}\quad\text{and}\quad A^*_{v^*}=\left\{(t,x):\,v^*(t,x)< \tilde{S}^*(t,x)\right\}.
$$
\begin{lemma}\label{v > h}
	Let $v:A\to\mathbb{R}$ and $ v^*:A\to\mathbb{R}$  be viscosity solutions to \eqref{Auxiliary boundary problem forward} and \eqref{Auxiliary boundary problem backward}. Then, $v$ and $v^*$ are, respectively, viscosity solutions of the boundary problems \eqref{Auxiliary boundary problem forward 2} and \eqref{Auxiliary boundary problem backward 2}. Additionally, $v(t,x)\leq \tilde S(t,x)$ and $v^*(t,x)\leq \tilde S^*(t,x)$. 
\end{lemma}
\begin{proof}
	Since the function $v$ is continuous, then for any ball ${\cal B}_{\epsilon}(t,x)$ with radius $\epsilon>0$ and center $(x,t)$, there exists $(t_0,x_0)$ such that $$v(t_0,x_0)=\max_{(t',x')\in\overline{\cal B}_{\epsilon}(t,x)} v(t',x').$$ 
	By choosing $\psi(t,x)=v(t_0,x_0)$, we have that $\psi(t_0,x_0)=v(t_0,x_0)$ and $(t_0,x_0)$ is a local minimum for the function $v-\psi$. Since the function $v$ is a subsolution to \eqref{Auxiliary boundary problem forward}, we have  $v(t_0,x_0)=\psi(t_0,x_0)\leq \tilde{S}(t_0,x_0)$. Letting $\epsilon$ go to $0$, we obtain $v(t,x)\leq \tilde{S}(t,x)$ for all $(t,x)\in A$. In particular $v(t,x)<\tilde{S}(t,x)$ when $(t,x)\in A_v$ and $v(t,x)=\tilde{S}(t,x)$ when $(t,x)\notin A_v$.
	
	Since $v(t,x)<\tilde{S}(t,x)$ when $(t,x)\in A_v$ and $v$ is a viscosity solution to \eqref{Auxiliary boundary problem forward}, we deduce that: (i) if $\psi\in C^2([-T/2,T/2)\times\mathbb{R}^n)$ and $(t,x)\in [-T/2,T/2)\times\mathbb{R}^n$ are such that ${v}-\psi$ has a local maximum at $(t,x)$ and ${v}(t,x)=\psi(t,x)$, then
	\begin{align*}
	&-\frac{\partial \psi}{\partial t}+\frac{1}{2}\vert\nabla \psi\vert^2-\frac{\hbar}{2}\Delta \psi-V(x)\leq 0;
	\end{align*} 
(ii) if $\psi\in C^2([-T/2,T/2)\times\mathbb{R}^n)$ and $(t,x)\in [-T/2,T/2)\times\mathbb{R}^n$ are such that ${v}-\psi$ has a local minimum at $(t,x)$ and ${v}(t,x)=\psi(t,x)$, then
\begin{align*}
&-\frac{\partial \psi}{\partial t}+\frac{1}{2}\vert\nabla \psi\vert^2-\frac{\hbar}{2}\Delta \psi-V(x)\geq 0;
\end{align*}
	which means that $v$ is a viscosity solution \eqref{Auxiliary boundary problem forward 2}. A similar argument can be used to prove the statements for $v^*$. 
\end{proof}
\vskip 3mm
 {From this result, it is clear that $U$ (resp., $U^*$) is a viscosity solution to the boundary problem \eqref{Auxiliary boundary problem forward 2} (resp., \eqref{Auxiliary boundary problem backward 2}), if one replaces $\tilde S$ by $S$ and $A_{v}$ by ${\cal C}$ (resp., $\tilde{S}^*$ by $S^*$ and $A_{v}^*$ by ${\cal C}^*$). Since the continuation regions $\cal C$ and $\cal C^*$ are unknown, initially, the boundary problems described above are known as free-boundary problems} (see for instance Caffarelli and Salsa \cite{caffarelli2005geometric}).

To finalize this section, we notice that, in light of the stochastic optimal control theory, (see, for instance Fleming and Soner \cite{fleming2006controlled}), we observe that the value functions  {$S_{\tau_A}$} and  {$S^*_{\tau^*_{A^*}}$}, defined by
\begin{align}\label{value function H}
& {S_{\tau_A}}(t,x)=\inf_{b\in {\cal U}}E_{t,x}\left[\int_{t}^{\tau_A}\{ \frac{1}{2}\vert b(u,Z_u )\vert^2 + V(Z_u ) \} du+S(Z_{\tau_A})\right]\\\label{value function H*}
& {S^*_{\tau^*_{A^*}}}(t,x)=\inf_{b^*\in {\cal U}^*}E_{t,x}\left[\int_{\tau^*_{A^*}}^{t}\{ \frac{1}{2}\vert b^*(s,Z_s)\vert^2 + V(Z_s)\} ds+S^*(Z_{\tau^*_{A^*}})\right]
\end{align}
satisfy, respectively, the boundary problems \eqref{Auxiliary boundary problem forward 2} and \eqref{Auxiliary boundary problem backward 2}, but now assuming that $A_v=A$ and $A_v=A^*$. In addition the optimal strategy is given by  {$\tilde b=-\nabla S_{\tau_A}$ and $\tilde b^*=\nabla S^*_{\tau^*_{A^*}}$}.   

\section{A uniqueness result}\label{A uniqueness result}

We now present a uniqueness result. We prove that our value functions are, indeed, the unique solutions for the boundary problems presented in \eqref{HJB v}-\eqref{HJB v*}-\eqref{boundary condition}. Additionally, in light of the Lemma \ref{v > h}, we deduce that $U$ and $U^*$ are solutions to  the boundary problems \eqref{Auxiliary boundary problem forward 2} and \eqref{Auxiliary boundary problem backward 2}, when we  replace, respectively, $A_v$ and $A_{v^*}^*$ by $\cal C$ and ${\cal C}^*$. Similar results in the field of optimal stopping can be found in {\O}ksendal and Reikvam \cite{reikvam1998viscosity}.

\begin{teo}\label{uniqueness result}
Consider the forward and backward stochastic optimal control problems defined respectively by \eqref{forward J}-\eqref{forward process}-\eqref{optimal U} and \eqref{backward J}-\eqref{backward process}-\eqref{optimal U*}. Then:
\begin{itemize}
\item[1)] the value function $U$ is the unique viscosity solution to the HJB equation \eqref{HJB v} that satisfies both the {left  boundary condition} in \eqref{boundary condition} and the condition
\begin{align}\label{uniformly integrability U}
&\{U({\tau ,}X_\tau)\}_{\tau\in {\cal P}_t} \text{ is a uniformly integrable} \text{ family of random variables.}
\end{align}
{Additionally the optimal strategy is given by $$\hat{\tau}=\inf\{ -T/2\leq t\leq u\leq T/2\,:U(u,Z_u )\geq S(Z_u )\}\text{ and }\hat b(t,x)=-\nabla U(t,x).$$}
\item[2)] the value function $U^*$ is the unique viscosity solution to the HJB equation  \eqref{HJB v*} that satisfies both the right-hand side {right boundary condition} in \eqref{boundary condition} and the condition
\begin{align}\label{uniformly integrability U*}
&\{U^*({\tau ,}X_\tau)\}_{\tau\in {\cal F}_t} \text{ is a uniformly integrable} \text{ family of random variables.}
\end{align}
{Additionally the optimal strategy is given by $$\hat{\tau}^*=\sup\{-T/2\leq s\leq t\leq T/2\,:U^*(s,Z_s)\geq S^*(Z_s)\}\text{ and }\hat b^*(t,x)=\nabla U^*(t,x).$$}
\end{itemize}
\end{teo}
\begin{proof}
 Consider an open bounded set $A_N\subset [-T/2,T/2]\times\mathbb{R}^n$ such that $A_N\nearrow [-T/2,T/2]\times\mathbb{R}^n$, as $N\to\infty$, and the function $v^*_N$ that verifies $v^*_N(t,x)=v^*(t,x), \text{ for all  }(t,x)\in\overline{A}_N,$ where $v^*$ is a viscosity solution to \eqref{HJB v*} such that the right hand-side of \eqref{boundary condition} is satisfied and $\{v^*(X_\tau)\}_{\tau\in {\cal F}_t}$ is a uniformly integrable family of random variables. By construction, we know that $v^*_N$ is a viscosity solution of \eqref{Auxiliary boundary problem backward}, when one fixes $\tilde S^*=v^*_N$ that is, in fact, unique according to the comparison principle for bounded domains, presented by Crandall, Ishii and Lions \cite{crandall1992user}. Therefore, from Proposition \ref{existence of viscosity solution} and Lemma \ref{remark}, we have 
 	\begin{align*}
 v^*_N(t,x)&=
 \inf_{(b^*,\tau^*)\in{\cal U}^*\times{\cal T}^*_{t}}E_{t,x}\left[\int_{\tau^*_N\vee\tau^*}^{t}\frac{1}{2}\vert b^*(s,Z_s)\vert^2 + V(Z_s) ds+v^*_N\left(\tau\vee\tau_{N}^*,Z_{\tau\vee\tau_{N}^*}\right)\right],
 \end{align*}
 where, {$\tau^*_N=\sup\{-T/2\leq s\leq t\leq T/2:(s,Z_s)\notin A_N\}$}. By construction, $b^*$ is already chosen (as one can see in $\eqref{optimal b}$), i.e ~$b^*=\tilde b^*\equiv-\nabla v^*$. Therefore,
	\begin{align*}
	v^*_N(t,x)&=
	\inf_{\tau^*\in{\cal T}^*_{t}}E_{t,x}\left[\int_{\tau^*_N\vee\tau^*}^{t} \frac{1}{2}\vert \tilde b^*(s,Z_s)\vert^2 + V(Z_s) ds+v^*_N\left(\tau\vee\tau_{N}^*,Z_{\tau\vee\tau_{N}^*}\right)\right]\\
	&\leq \inf_{\tau^*\in{\cal T}^*_{t}}E_{t,x}\left[\int_{\tau^*_N\vee\tau^*}^{t} \frac{1}{2}\vert \tilde b^*(s,Z_s)\vert^2 + V(Z_s) ds+S^*\left(Z_{\tau\vee\tau_{N}^*}\right)\right]
	\end{align*}
	the last inequality being a consequence of Lemma \ref{v > h}.
%	where $\tau^*_{N}\equiv\sup\{s<t:(s,Z_s)\notin A_N\}$, $(\tilde\tau_N^*,\tilde b^*)$ is the optimal strategy, and $\tilde b^*$ is defined as in \eqref{optimal b}.  
	Since $A_N\nearrow I\times [0,\infty)$ as $N\to\infty$, then $\tau^*_N\vee\tau^*\searrow-T/2\vee\tau^*$. Additionally, 
	$$0\leq \int_{\tau^*_N\vee\tau^*}^{t}\left(\frac{1}{2}\vert \tilde b^*(s,Z_s)\vert^2 + V(Z_s)\right)^\pm ds\nearrow\int_{\tau^*_N\vee\tau^*}^{t}\left(\frac{1}{2}\vert \tilde b^*(s,Z_s)\vert^2 + V(Z_s)\right)^\pm ds$$
	where $\left(\frac{1}{2}\vert \tilde b^*(s,Z_s)\vert^2 + V(Z_s)\right)^+=\max\left(0,\frac{1}{2}\vert \tilde b^*(s,Z_s)\vert^2 + V(Z_s)\right)$ and $\left(\frac{1}{2}\vert \tilde b^*(s,Z_s)\vert^2 + V(Z_s)\right)^-=\max\left(0,-\left(\frac{1}{2}\vert \tilde b^*(s,Z_s)\vert^2 + V(Z_s)\right)\right)$. From the monotone convergence theorem, we obtain
	$$
	\lim_{n\to\infty}E_{t,x}\left[\int_{\tau^*_N\vee\tau^*}^{t} \frac{1}{2}\vert \tilde b^*(s,Z_s)\vert^2 + V(Z_s) ds\right]=E_{t,x}\left[\int_{-T/2\vee\tau^*}^{t}  \frac{1}{2}\vert \tilde b^*(s,Z_s)\vert^2 + V(Z_s)  ds\right]
	$$ 
	Furthermore, $\{S^*(Z_{\tau^*})\}_{\tau^*}$ is a uniformly integrable family of random variables, which implies that
	$$
	\lim_{N\to\infty}E_{t,x}\left[S^*(Z_{\tau^*_N\vee\tau^*})\right]=E_{t,x}\left[S^*(Z_{-T/2\vee\tau^*})\right].
	$$
%	Furthermore, since the function $v_N^*(\tau^*,Z_{\tau^*})$ is a uniformly integrable family of random variables  (because $v^*_N(\tau^*,Z_{\tau^*})$ is a uniformly integrable family of random variables) and $\lim\limits_{N\to\infty}v_N^*(t,x)=v^*(t,x)$, then
%	$$
%	\lim_{n\to\infty}E_{t,x}\left[v_N^*(\tilde\tau^*_N,Z_{\tilde\tau^*_N})\right]=E_{t,x}\left[v^*(-T/2\vee\tilde\tau^*,Z_{-T/2\vee\tilde\tau^*})\right].
%	$$
	Since this holds true for every $\tau^*\in{\cal T}_t^*$, we have
	\begin{align*}
	v^*(t,x)=&\lim_{N\to+\infty}v_N^*(t,x)\leq \lim_{N\to+\infty}U^*_N(t,x)=U^*(t,x).
	\end{align*}
	
	To prove the result, one still needs to show that $	v^*(t,x)\geq U^*(t,x).$ Let $A^N_{v^*}=\{(t,x)\in A_N:v^*(t,x)<S^*(x)\}$ and {$\tilde\tau^*_N=\sup\{-T/2\leq s\leq t\leq T/2:(s,Z_s)\notin A^N_{v^*}\}$}. Combining the first part of this proof with the results in Lemma \ref{v > h} it follows that $v^*_N$ is the unique viscosity solution of \eqref{Auxiliary boundary problem backward 2}. Additionally, in light of our discussion regarding the representation of the value function for the control problems \eqref{value function H} and \eqref{value function H*}, we obtain
	\begin{align*}
	v^*_N(t,x)&=E_{t,x}\left[\int_{\tilde\tau^*_N}^{t} \frac{1}{2}\vert \tilde b^*(s,Z_s)\vert^2 + V(Z_s)  ds+v^*_N\left(\tilde\tau^*_N,Z_{\tilde\tau^*_N}\right)\right].
	\end{align*}
	Noticing that $A_N\nearrow \{(t,x)\in [-T/2,T/2]\times\mathbb{R}^n:v^*(t,x)<S^*(x)\}$ as $N\to\infty$, then $\tilde\tau^*_N=\tau^*_v\vee\tau^*_N\searrow-T/2\vee\tau^*_v$, where $\tau^*_v=\sup\{s<t:(s,Z_s)\notin [-T/2,T/2]\times\mathbb{R}^n\}$. Therefore,  using a similar argument to the previous one, we get $v(t,x)=\lim_{N\to+\infty}v_N(t,x)$, and, consequently,
	\begin{align*}
	v(t,x)&=E_{t,x}\Bigg[\int_{-T/2\vee\tau^*_v}^{t} \frac{1}{2}\vert \tilde b^*(s,Z_s)\vert^2 + V(Z_s) ds+v^*_N\left(-T/2\vee\tau^*_v,Z_{-T/2\vee\tau^*_v}\right)\Bigg]\\
	&=E_{t,x}\left[\int_{-T/2\vee\tau^*_v}^{t}  \frac{1}{2}\vert \tilde b^*(s,Z_s)\vert^2 + V(Z_s) ds+S^*\left(Z_{-T/2\vee\tau^*_v}\right)\right]\geq U^*(t,x).
	\end{align*}
	
	From this argument, it follows that $U^*=v^*$, that is unique, $\hat{\tau}^*=\tau^*_v$ and $\hat b=\nabla v^*=\nabla U^*$. The argument to prove the result for $U$ would be very similar to the one presented here.
\end{proof}
\vskip 3mm
The usual construction of Bernstein stochastic processes relies on the solution of a forward and backward heat equations with respective positive (not necessarily integrable)  final and initial conditions. In our case, we can also construct this class of diffusion processes following a similar strategy. 

Let $\eta$ and $\eta^*$ be two functions defined as follows:
$\eta(t,x)=e^{-\frac{1}{\hbar}U(t,x)}$ and $\eta^*(t,x)=e^{-\frac{1}{\hbar}U^*(t,x)}$. One can check  that $\eta$ and $\eta^*$ satisfy, respectively, the following boundary problems
\begin{align*}
&\begin{cases}
&\min\left\{-\hbar\frac{\partial \eta}{\partial t}-\frac{\hbar^2}{2}\Delta \eta {+}V(t,x)\eta,\eta-e^{-\frac{1}{\hbar}S(x)}\right\}=0\\
&\eta\left(\frac{T}{2},x\right)=e^{-\frac{1}{\hbar}S(x)},
\end{cases}\\
&\begin{cases}
&\min\left\{\hbar\frac{\partial \eta^*}{\partial t}-\frac{\hbar^2}{2}\Delta \eta^*+V(t,x)\eta^*,\eta^*-e^{-\frac{1}{\hbar}S^*(x)}\right\}=0\\
&\eta^*\left(-\frac{T}{2},x\right)=e^{-\frac{1}{\hbar}S^*(x)}.
\end{cases}
\end{align*}
The reverse is also true in the sense that if $\eta$ and $\eta^*$ are positive functions then $v(t,x)=-\hbar(\log ~\eta)(t,x)$ and $v^*(t,x)=-\hbar(\log ~\eta^*)(t,x)$ solve Eqs  \eqref{HJB v} and \eqref{HJB v*}. Furthermore, as a consequence of Lemma \ref{v > h}, we also know that $\eta$ and $\eta^*$ satisfy, respectively, the following boundary problems:
\begin{align*}
\begin{cases}
	-\hbar\frac{\partial \eta}{\partial t}-\frac{\hbar^2}{2}\Delta \eta {+}V(t,x)\eta=0,&(t,x)\in {\cal C}\\
	\eta\left(t,x\right)=e^{-\frac{1}{\hbar}S(x)},&(t,x)\in \partial {\cal C}
\end{cases} \quad
\begin{cases}
	\hbar\frac{\partial \eta^*}{\partial t}-\frac{\hbar^2}{2}\Delta \eta^*+V(t,x)\eta^*=0,&(t,x)\in {\cal C}^*\\
	\eta^*\left(t,x\right)=e^{-\frac{1}{\hbar}S^*(x)},& (t,x)\in \partial{\cal C}^*.
\end{cases}
\end{align*}
{To end this section, we note that when $U$ and $U^*$ are smooth enough then the controlled drifts are, respectively, $\hat b(t,x)=\hbar\nabla\log(\eta(t,x))$ and $\hat b^*(t,x)= {-}\hbar\nabla\log(\eta^*(t,x))$. Additionally, $\hat{b}$ and $\hat{b}^*$ solve the boundary problems:}
\begin{align*}
&\begin{cases}
\frac{\partial \hat b}{\partial t}+(\hat b .\nabla) \hat b +\frac{\hbar^2}{2}\Delta \hat b-\nabla V(t,x)=0,& (t,x)\in {\cal C}\\
\hat b \left(t,x\right)=-\nabla S(x),\qquad (t,x)\in \partial{\cal C}
\end{cases}\\
&\begin{cases}
\frac{\partial \hat b^*}{\partial t}+(\hat b^* .\nabla) \hat b^* -\frac{\hbar^2}{2}\Delta \hat b^*-\nabla V(t,x)=0,&(t,x)\in {\cal C}^*\\
\hat b^* \left(t,x\right)=\nabla S^*(x),\qquad (t,x)\in \partial{\cal C}^*
\end{cases}.
\end{align*}

\section{Characterization of the optimal times}
In this section, we are interested in  obtaining a full characterization of the optimal stopping times $\hat{\tau}$ and $\hat\tau^*$ used in the previous section.

To characterize the distribution of these stopping times, one has to provide a characterizations of the following functions:
\begin{align}\label{function u}
q(t,x)&=P_{t,x}(\hat\tau>\tilde T),\text{ with } (t,x)\in [-T/2,T/2]\times \mathbb{R}^n\text{ and }-T/2\leq \tilde{T}\leq T/2,\\\label{function u*}
q^*(t,x)&=P_{t,x}(\hat\tau^*<\tilde T),\text{ with } (t,x)\in [-T/2,T/2]\times \mathbb{R}^n\text{ and }-T/2\leq \tilde{T}\leq T/2. 
\end{align}
One may notice that  functions $q$ and $q^*$ can be written in the following way:
\begin{align}\label{function u - expected value}
q(t,x)&=E_{t,x}\left[g({\hat\tau\wedge \tilde{T}},Z_{\hat\tau\wedge \tilde{T}})\right],\text{ with } -T/2\leq t\leq \tilde{T}< T/2\text{ and }(t,x)\in{\cal C},\\
q^*(t,x)&=E_{t,x}\left[g^*(\hat\tau^*\wedge \tilde{T},Z_{\hat\tau^*\wedge \tilde{T}})\right],\text{ with } -T/2<\tilde T\leq t\leq {T}/2\text{ and }(t,x)\in{\cal C}^*,
\end{align} 
where the function  {$g$ and $g^*$} are defined by:
$$
g(s,x)=1_{\{U(s,x)<S(x)\}}\quad\text{and}\quad g^*(s,x)=1_{\left\{U^*(s,x)<S^*(x)\right\}}.
$$
This follows from the fact that 
\begin{align}
\{\tilde\tau>\tilde T\}&=\left\{\omega\in \Omega\,:U(s,Z_s {(\omega ) } )<S(Z_s  {(\omega )}), \forall s\in[t,\tilde T]\right\},\\
\{\tilde\tau^*<\tilde T\}&=\left\{\omega\in \Omega\,:U^*(s,Z_s  {(\omega )})< {S^*(Z_s  {(\omega )})}, \forall s\in[\tilde T,t]\right\},
\end{align}
which is obvious by the definitions of $\{\tilde\tau>\tilde T\}$ and $\{\tilde\tau^*<\tilde T\}$.

The results derived below require to assume some regularity on the controls. 
\begin{assumption}\label{assumption optimal time}
The function $U:[-T/2,T/2]\times\mathbb{R}^n\to\mathbb{R}$ and $U^*:[-T/2,T/2]\times\mathbb{R}^n\to\mathbb{R}$, defined  in \eqref{optimal U} and \eqref{optimal U*}, are such that the following stochastic differential equations have a unique  strong solution 
\begin{align}\label{forward sde}
dZ_u&=-\nabla U(u,Z_u )du+\hbar^{1/2}dW_u ,\quad Z_t=x\text{ and }-\frac{T}{2}\leq t\leq u\leq \frac{T}{2}\\\label{backward sde}
d_*Z_s&=\nabla U^*(s,Z_s)d_*s+\hbar^{1/2}d_*W^*_s , \quad Z_t=x\text{ and }-\frac{T}{2}\leq s\leq t\leq \frac{T}{2}.
\end{align} 
Additionally, there are constants $K>0$ and $K^*>0$ such that
\begin{align*}
\vert \nabla U(s,x)-\nabla U(s,y)\vert\leq K\vert x-y\vert,\text{ for all } (s,x) \text{ and }(s,y)\in{\cal C}\\\ \vert\nabla U^*(s,x)-\nabla U^*(s,y)\vert\leq K^*\vert x-y\vert,\text{ for all } (s,x) \text{ and }(s,y)\in{\cal C}^*
\end{align*}
\end{assumption}

In some cases the distribution probabilities above are {easy to evaluate}, as one may see in the next lemma, stated without proof. It will be useful to introduce the following notation 
$$
\overline{t}=\sup_t\{t\in[-T/2,T/2]\,:(t,x)\in{\cal C}\}\text{ and }\underline{t}=\inf_t\{t\in[-T/2,T/2]\,:(t,x)\in{\cal C}^*\}.
$$
for some $x\in\mathbb{R}^n$.
\begin{lemma}\label{first u and u* characterization}
	Let $(t,x)\in[-T/2,T/2]\times\mathbb{R}^n$ and $-T/2\leq \tilde T\leq T/2$. Then function $q$ verifies
	\begin{itemize}
		\item[i)] $q(t,x)=1$ if $
		\left((t,x)\in{\cal S}\text{ and }t>\tilde{T}\right)\text{ or } 	\left((t,x)\in{\cal C}\text{ and }t\geq\tilde{T}\right)
		$; 
		\item[ii)] $q(t,x)=0$ if $
		\left((t,x)\in{\cal S}\text{ and }t\leq\tilde{T}\right)\text{ or } 	\left((t,x)\in{\cal C}\text{ and }\tilde{T}\geq \overline{t}\right).
		$
		\end{itemize}
	Regarding  function $q^*$ symmetric statements can be obtained:
	\begin{itemize}
	\item[iii)] $q^*(t,x)=1$ if 
	$
	\left((t,x)\in{\cal S}^*\text{ and }t<\tilde{T}\right)\text{ or } 	\left((t,x)\in{\cal C}^*\text{ and }t\leq\tilde{T}\right)
	$
	\item[iv)] $q^*(t,x)=0$ if
	$
	\left((t,x)\in{\cal S}^*\text{ and }t\geq\tilde{T}\right)\text{ or } 	\left((t,x)\in{\cal C}^*\text{ and }\tilde{T}\leq\underline{t}\right).
	$
\end{itemize}
\end{lemma}
%When $\tilde T=T/2$ and $\tilde T^*=T/2$ then we have that $P_{t,x}(\hat\tau>\tilde T)=P_{t,x}(\hat\tau^*<\tilde T^*)=1$. Moreover, if $t>\tilde T$ then $P_{t,x}(\hat\tau>\tilde T)=1$ and, similarly, when $t<\tilde T^*$, then $P_{t,x}(\hat\tau^*<\tilde T^*)=1$. Therefore, one still needs to characterize $u$ when $-T/2\leq t\leq \tilde{T}< T/2\text{ and }x\in{\cal C}$ and $u^*$ when $-T/2<\tilde T^*\leq t\leq {T}/2\text{ and }x\in{\cal C}^*$.
In the remaining cases, we will show that, under additional conditions, functions $q$ and $q^*$ are the unique continuous viscosity solutions  of the following boundary problems:
\begin{align}\label{boundary problem u}\begin{cases}
&\frac{\partial q}{\partial t}-\nabla U(t,x)\cdot\nabla q+\frac h 2\Delta q=0,\, \text{ with }-T/2\leq t< \tilde{T}< \overline{t}\text{ and }(t,x)\in{\cal C}\\
&q(\tilde T,x)=1,\,\text{ for  }( {\tilde T},x)\in {\cal C}\\
&q(\tilde t,\tilde x)=0,\,\text{ for  } (\tilde t,\tilde x)\in \{(t,x)\,:U(t,x)=S(x)\wedge -T/2\leq t<\tilde T<\overline{t}\}
\end{cases},
\end{align}
and
\begin{align}\label{boundary problem u*}\begin{cases}
&\frac{\partial q^*}{\partial t} {+}\nabla U^*(t,x)\cdot\nabla q^*-\frac{h}{2}\Delta q^*=0,\, \text{ with } \underline{t}< \tilde{T}<t\leq T/2\text{ and }(t,x)\in{\cal C}^*\\
&q^*(\tilde T,x)=1,\,\text{ for  }(\tilde T,x)\in {\cal C}^*\\
&q^*(\tilde t,\tilde x)=0, \, \text{ for  } (\tilde t,\tilde x)\in \{(t,x)\,:U^*(t,x)=S^*(x)\wedge\, \underline{t}< \tilde{T}<t\leq T/2\}
\end{cases}.
\end{align}

Note that there is an implicit relationship between $q$ and $q^*$ since $\nabla U^*(t,x)=-\nabla U(t,x)-\hbar\nabla\log\rho(t,x)$, where $\rho$ represents the density of the process.

Let us observe that, by definition of the optimal drifts $b$ and $\hat b$ in (16), $q$ and $q^*$ are respectively a ${\cal P}_t$ - martingale and a ${\cal F}_t$ - martingale of the process $Z$.

\begin{prop}\label{prop continuity}
	Let $q$ and $q^*$ be the functions defined  in \eqref{function u} and \eqref{function u*}. Then, $q$ is continuous in the domain $(t,x)\in{\cal C}$ and $-T/2\leq t<\tilde T<\overline{t}$ and $q^*$ in $(t,x)\in{\cal C}^*$ and $\underline{t}< \tilde{T}<t\leq T/2$.
\end{prop}
To prove this proposition, we will first state some auxiliary results.

In the next result we present some estimates on the moments of the process $Z$. This allow us prove the continuity of the application $(s,t,x)\to Z^{t,x}_s(\omega)$. The result's proof will use standard arguments and, consequently, we will simply draft the proof highlighting the more relevant steps. 

\begin{lemma}
	Let $Z$ be the Bernstein process satisfying the forward and backward stochastic differential equations \eqref{forward sde} and \eqref{backward sde}. 
	Then, for each fixed $\omega\in \Omega$, the application $(s,t,x)\to Z^{t,x}_s(\omega)$ is continuous for each $(t,x)\in{\cal C}$ and $t\leq s<\hat{\tau}$.
\end{lemma}
\begin{proof}
	We start the proof by noticing that
	\begin{align*}
	E\left[\left\vert Z^{t,x}_s-Z^{t',x'}_{s'}\right\vert^p\right]&\leq 3^{p-1}\left(E\left[\left\vert Z^{t,x}_s-Z^{t,x}_{s'}\right\vert^p\right]+E\left[\left\vert Z^{t,x}_{s'}-Z^{t,x'}_{s'}\right\vert^p\right]+E\left[\left\vert Z^{t,x'}_{s'}-Z^{t',x'}_{s'}\right\vert^p\right]\right).
	\end{align*}
	Firstly we will prove that
	$$
	E\left[\left\vert Z^{t,x}_s-Z^{t,x}_{s'}\right\vert^p\right]\leq 2^{p-1}(s'-s)^\frac{p}{2}\left(K^pT^\frac{p}{2}+\left(\frac{p(p-1)}{2}\right)^{\frac{p}{2}}{\hbar}^p\right).
	$$
	To prove this estimate, one can notice that
	\begin{align*}
	&E\left[\left\vert Z^{t,x}_s-Z^{t,x}_{s'}\right\vert^p\right]\leq 2^{p-1}E\left[\left\vert\int_s^{s'}\nabla U(u,Z^{t,x}_u)du\right\vert^p+\left\vert\int_s^{s'}\hbar dW_u\right\vert^p\right]\\
	&\leq 2^{p-1}\left(L^p(s'-s)^p+\left(\frac{p(p-1)}{2}\right)^{\frac{p}{2}}{\hbar}^p(s'-s)^\frac{p}{2}\right)\leq 2^{p-1}(s'-s)^\frac{p}{2}\left(L^pT^\frac{p}{2}+\left(\frac{p(p-1)}{2}\right)^{\frac{p}{2}}{\hbar}^p\right),
	\end{align*}
	the first inequality following from H\"older's inequality, Theorem 1.7.1 in Mao \cite{mao2007stochastic} and the fact that $U$ is Lipschitz in $x$ uniformly in $t$, which implies that $\vert \nabla U(s,x)\vert $ is bounded by a constant $L>0$, for every $(s,t)\in {\cal C}$.
	
	To find the estimate
	\begin{align}\label{estimate 2}
	E\left[\left\vert Z^{t,x}_{s'}-Z^{t,x'}_{s'}\right\vert^p\right]\leq 2^{p-1}(x-x')^pe^{K^p(t-s')^{p}}\leq 2^{p-1}(x-x')^pe^{\frac{(2KT)^p}{2}}
	\end{align}
	one may notice that
	\begin{align*}
	&E\left[\left\vert Z^{t,x}_{s'}-Z^{t,x'}_{s'}\right\vert^p\right]\leq 2^{p-1}\left((x-x')^p+E\left[\left\vert\int_{s'}^{t}\nabla U(u,Z^{t,x}_u)-\nabla U(u,Z^{t,x'}_u)du\right\vert^p\right]\right)\\
	&\leq 2^{p-1}\left((x-x')^p+K^p(t-s')^{p-1}E\left[\int_{s'}^{t}\left\vert Z^{t,x}_u-Z^{t,x'}_u\right\vert^pdu\right]\right),
	\end{align*}
	where  H\"older's inequality has been used as well as Theorem 1.7.1 in Mao \cite{mao2007stochastic} and Assumption \ref{assumption optimal time}. By   Gronwall's inequality, this shows  estimate \eqref{estimate 2}.
	
	Finally, along the same lines of the previous estimates, we may prove that 
	\begin{equation}\label{estimate 3}
	E\left[\left\vert Z^{t,x'}_{s'}-Z^{t',x'}_{s'}\right\vert^p\right]\leq 2^{p-1}L^p(t'-t)e^{\frac{(2TK)^p}{2}}.
	\end{equation}
	
	The result follows from Kolmogorov's Lemma (see, for instance Theorem 72, Chapter IV in Protter \cite{Protter}).
\end{proof}
\vskip 3mm

\begin{lemma}\label{lemma continuities}
	Let $\hat\tau$ and $\hat\tau^*$ be the stopping times defined  in Theorem \ref{uniqueness result}. Then, for  fixed $\omega\in\Omega$, $(t,x)\to\tau_{t,x}(\omega)$ is a continuous application in the domain $(t,x)\in{\cal C}$ and $-T/2\leq t<\tilde T<\overline{t}$ and $(t,x)\to\tau^*_{t,x}(\omega)$ is a continuous application in $(t,x)\in{\cal C}^*$ and $\underline{t}< \tilde{T}<t\leq T/2$.
\end{lemma}
\begin{proof}
	To prove that the application $(t,x)\to\tau_{t,x}(\omega)$ is continuous for a fixed $\omega\in\Omega$, we notice that, due to the continuity of $x\to U(s,x)-S(x)$ and $(t,x)\to X_s^{t,x}$, for any $s\in[-T/2,T/2]$, we get that for all $\epsilon>0$ there exists $\delta>0$ such that 
	\begin{align}\label{continuity of U-S}
	\vert t-t'\vert+\vert x-x'\vert<\delta\Rightarrow\vert U(s,X_s^{t,x})-S(X_s^{t,x})-U(s,X_s^{t',x'})+S(X_s^{t',x'})\vert<\epsilon.
	\end{align}
	According to the definition of  $\hat\tau$, we get that for all $\gamma>0$ there is $\zeta>0$ such that
	\begin{align}\label{inequality for tau t,x}
	&U(s,X_s^{t,x})-S(X_s^{t,x})<-\zeta,\text{ for all }s\in[t,\hat\tau_{t,x}-\gamma]\\\label{inequality for tau t',x'}
	&U(s,X_s^{t',x'})-S(X_s^{t',x'})<-\zeta,\text{ for all }s\in[t,\hat\tau_{t',x'}-\gamma].
	\end{align}
	Thus, combining \eqref{continuity of U-S} with \eqref{inequality for tau t,x}  and choosing $\epsilon=\frac \zeta 2$, we have
	$$
	U(s,X_s^{t',x'})-S(X_s^{t',x'})-\frac \zeta 2<U(s,X_s^{t,x})-S(X_s^{t,x})<-\zeta.
	$$
	Therefore, $U(s,X_s^{t',x'})-S(X_s^{t',x'})<-\frac \zeta 2$ for all $s\in[t',\hat\tau^{t,x}-\gamma]$ which implies that $\tau^{t,x}-\gamma<\tau^{t',x'}$. By combining \eqref{continuity of U-S} with \eqref{inequality for tau t',x'} and using a similar argument, we conclude that $\tau^{t',x'}-\gamma<\tau^{t,x}$. Therefore, we have proved that for all $\gamma>0$ there is $\delta>0$ such that 
	\begin{align*}
	\vert t-t'\vert+\vert x-x'\vert<\delta\Rightarrow\vert \tau_{t,x}-\tau_{t',x'}\vert<\gamma,
	\end{align*}
	as required.
\end{proof}
\vskip 3mm

\begin{proof}[Proof of Proposition \ref{prop continuity}]
	In what follows, we prove that $(t,x)\to q(t,x)$ is a continuous function in the domain $-T/2\leq t< \tilde{T}< 
	\overline{t}\text{ and }(t,x)\in{\cal C}$. A similar argument may be established for the remaining case.
	
Fix $\omega\in\Omega$ such that $\hat\tau_{(t',x')}>\tilde{T}$ or $\hat\tau_{(t',x')}<\tilde{T}$. Due to the continuity of the functions $(t,x)\to\hat\tau_{t,x}$ and $(s,t,x)\to Z_s^{t,x}$, one has { for $g$ of Equation \eqref{function u - expected value},}
	$$
	\lim_{(t,x)\to(t',x')}g(\hat\tau_{t,x}\wedge\tilde{T},Z_{\hat\tau_{t,x}\wedge\tilde{T}}^{t,x})=g(\hat\tau_{t',x'}\wedge\tilde{T},Z_{\hat\tau_{t',x'}\wedge\tilde{T}}^{t',x'}).
	$$
Additionally, since the drifts of the process $Z$ are bounded, as noticed in the proof of Lemma \ref{lemma continuities}, Girsanov theorem holds true, and consequently the law of $Z$ is absolutely continuous with respect to law of the Brownian motion. Therefore, combining this fact with the continuity of $U$ and $S$, we get that
	$$
	P(\hat\tau_{(t',x')}=\tilde{T})\leq P(U(\tilde{T},Z_{\tilde{T}})=S(Z_{\tilde{T}}))=0.
	$$
	Since by definition $0\leq g(t,x)\leq 1$ for all $(t,x)\in[-T/2,T/2]\times\mathbb{R}^n$, 
	$$
		\lim_{(t,x)\to(t',x')}E\left[g(\hat\tau_{t,x}\wedge\tilde{T},Z_{\hat\tau_{t,x}\wedge\tilde{T}}^{t,x})\right]=E\left[g(\hat\tau_{t',x'}\wedge\tilde{T},Z_{\hat\tau_{t',x'}\wedge\tilde{T}}^{t',x'})\right],
	$$
	 follows from the dominated convergence theorem.
\end{proof}
\vskip 3mm

%In fact, one can recall that $$\hat\tau=\inf\{s\geq t\,:U(s,Z_s)\geq S(Z_s)\} \text{ and } \hat\tau^*=\sup\{s\in]-T/2,t]\,:U^*(s,Z_s)\geq S^*(Z_s)\},$$ 
%which implies that
%$$
%\{\tilde\tau>\tilde T\}=\left\{\omega\in \Omega\,:U(s,X_s)<S(X_s), \forall s\in[t,\tilde T]\right\},
%$$
%and 
%$$
%\{\tilde\tau^*<\tilde T^*\}=\left\{\omega\in \Omega\,:U^*(s,X_s)<S(X_s), \forall s\in[\tilde T^*,t]\right\},
%$$
%and, therefore, the result follows immediately. 

%In the next proposition, we will show that, indeed, the functions $u$ and $u^*$ satisfy the following boundary problems:
%\begin{align}
%&\frac{\partial u}{\partial t}+b^*(t,x)\cdot\nabla u+\frac h 2\Delta u=0, \quad \text{with }-T/2\leq t< \tilde{T}< T/2\text{ and }x\in{\cal C}\\
%&u(\tilde T,x)=1,\text{ with }x\in {\cal C}\\
%&u(\tilde t,\tilde x)=0, \text{ for all } (\tilde t,\tilde x)\in \{(t,x)\,:U(t,x)=S(x)\wedge -T/2\leq t<\tilde T<T/2\}.
%\end{align}
%and
%\begin{align}
%&\frac{\partial u^*}{\partial t}+b^*(t,x)\cdot\nabla u^*-\frac{h}{2}\Delta u^*=0 \quad \text{with }-T/2\leq t< \tilde{T}< T/2\text{ and }x\in{\cal C}\\
%&u^*(\tilde T^*,x)=1,\text{ for all }x\in {\cal C}^*\\
%&u^*(\tilde t,\tilde x)=0, \text{ for all } (\tilde t,\tilde x)\in \{(t,x)\,:U^*(t,x)=S^*(x)\wedge -T/2<\tilde T^*<t\leq T/2\}.
%\end{align}
\begin{teo}\label{theorem optimal times}
Let $q$ and $q^*$ be the functions defined  in \eqref{function u} and \eqref{function u*}. Then, $q$ is the unique continuous viscosity solution of the boundary problem \eqref{boundary problem u} in the domain $(t,x)\in{\cal C}$ with $-T/2\leq t<\tilde{T}<\overline{t}$ and $q^*$ is the unique continuous viscosity solution to the boundary problem \eqref{boundary problem u*} in $(t,x)\in{\cal C}^*$ with $\underline{t}<\tilde{T}^*<t\leq T/2$. Outside of this domain domain, function $q$ and $q^*$ are characterized according to Lemma \ref{first u and u* characterization}. 
\end{teo}

Although the proof of Theorem \ref{theorem optimal times} relies on the same type of arguments used through out Sections \ref{Dynamic programming principle and viscosity solutions} and \ref{A uniqueness result}, we will shortly prove the result.  
\begin{proof}
		We will only consider the function $q$; regarding $q^*$, the statement can be proved along the same lines. We split the proof in two steps: (i) existence of solution to \eqref{boundary problem u} and (ii) uniqueness  of solution to \eqref{boundary problem u}. 

\textit{Proof of (i):}
Let $(t,x)\in{\cal C}$ and $-T/2\leq t<\tilde{T}<\overline{t}$ and $\psi\in C^2([-T/2,T/2]\times\mathbb{R}^n)$ be such that $(t,x)$ is a local maximizer of $q-\psi$ and $q(t,x)-\psi(t,x)=0$. Let $\tau\in{\cal T}_t$ be a stopping time satisfying $\tau\leq \hat\tau$, then, by the strong Markov property { and \eqref{function u - expected value}}  we get
\begin{align}
\psi(t,x)&=q(t,x)=E_{t,x}\left[E_{\tau,X_\tau}\left[g({\hat\tau\wedge \tilde{T}},Z_{\hat\tau\wedge \tilde{T}})\right]\right]=E_{t,x}\left[q(\tau,X_\tau)\right]\\
&\leq E_{t,x}\left[\psi(\tau,X_\tau)\right]=\psi(t,x)+E_{t,x}\left[\int_t^\tau{\cal L}\psi(s,X_s)ds\right].
\end{align}
Dividing the inequality $0\leq E_{t,x}\left[\int_t^\tau{\cal L}\psi(s,X_s)ds\right]$ by $E_{t,x}[\tau]$ and letting $\tau\searrow t$ we obtain that
$$
{\cal L}\psi(t,x)\geq 0,
$$
that allows us to conclude that $u$ is a viscosity subsolution to the PDE  \eqref{boundary problem u}.
To prove the viscosity supersolution a similar argument may be used, namely, pic $\psi\in C^2([-T/2,T/2]\times\mathbb{R}^n)$ and $(t,x)\in{\cal C}$ with $-T/2\leq t<\tilde{T}<\overline{t}$ such that $(t,x)$ is a local minimizer of $u-\psi$ and $q(t,x)-\psi(t,x)=0$. Then,
\begin{align}
0\geq E_{t,x}\left[\int_t^\tau{\cal L}\psi(s,X_s)ds\right],
\end{align}
for $\tau\in{\cal T}_t$ with $\tau\leq \hat\tau$. Dividing the last expression by $E_{t,x}[\tau]$ and letting $\tau\searrow t$ we obtain that
$$
{\cal L}\psi(t,x)\leq 0.
$$
Therefore, $q$ is a viscosity supersolution to the PDE  \eqref{boundary problem u}.

Finally, to prove that $q$ is a viscosity solution of the boundary problem \eqref{boundary problem u}, one can see that, in light of Lemma \ref{first u and u* characterization}, $q(\tilde{T},x)=1$ for all $(T,x)\in{\cal C}$. Additionally, it is straightforward that, if $(\tilde t,\tilde x)\in \{(t,x)\,:U(t,x)=S(x)\wedge -T/2\leq t<\tilde T< \overline{t}\}$, then $q(\tilde{t},\tilde{x})=P(\hat{\tau}>\tilde{T})=0$.

\textit{Proof of (ii):} Let $A_N$ be an open bounded set such that $$A_N\subset\{(t,x)\in{\cal C}\,:-T/2\leq t<\tilde T<\overline{t}\}\text{ and } A_N\nearrow\{(t,x)\in{\cal C}\,:-T/2\leq t<\tilde T<\overline{t}\}$$ and $\tau_N=\inf\{-T/2\leq t\leq u\,:Z_u \notin A_N\}$. Additionally, let $q_N$ be given by the function $q_N(t,x)=q(t,x)$ for all $(t,x)\in \overline{A}_N$, where $q$ is a viscosity solution to \eqref{boundary problem u}. By construction, $q_N$ is a viscosity solution of the boundary problem
\begin{align}\label{auxiliar boundary problem for u}
{\cal L}v=0\quad\text{with}\quad v=q_N.
\end{align}
Additionally, by using the comparison principle for bounded domains, presented by Crandall, Ishii and Lions \cite{crandall1992user}, one may conclude that $q_N$ is the unique viscosity solution of \eqref{auxiliar boundary problem for u}. 

Along the same lines as the first part of this proof, we have 
$$
q_N(t,x)=E\left[q_N\left(\tau_{N}\wedge\tilde{T},Z_{\tau_{N}\wedge\tilde{T}}^{t,x}\right)\right].
$$

Since $0\leq q_N(t,x)\leq 1$ for all $(t,x)\in \{(t,x)\in{\cal C}\,:-T/2\leq t<\tilde T<T/2\}$, the dominated convergence theorem allows us to conclude that
$$
q(t,x)=\lim_{N\to +\infty}q_N(t,x)=\lim_{N\to +\infty}E\left[q_N\left(\hat\tau_{N}\wedge\tilde{T},Z_{\hat\tau_{N}\wedge\tilde{T}}^{t,x}\right)\right]=E\left[q\left(\hat\tau_{t,x}\wedge\tilde{T},Z_{\hat\tau_{t,x}\wedge\tilde{T}}^{t,x}\right)\right],
$$
the last equality following from the continuity of $q$ and the fact that $\tau_N\nearrow\hat\tau_{t,x}$. Therefore, it is straightforward that $q$ is given by \eqref{function u}.
\end{proof}
\vskip 3mm

\section{Example}
In this section, we intend to solve the stochastic control problems \eqref{optimal U} and \eqref{optimal U*} when $n=1$, $V(x)=0$, $S(x)=\vert x\vert$ and $S^*(x)=\log(\vert x\vert+1)$, for all $x\in\mathbb{R}$. Additionally, we will compare the value functions and the Bernstein process obtained with our procedure with the ones  we would obtain solving the simpler (deterministic time interval) problem
\begin{align}\label{forward H}
\tilde H(t,x)=\inf_{b\in{\cal U}}E_{t,x}\left[\int_{t}^{T/2}(b(s,Z_s))^2  ds+\vert Z_{\tau\wedge T/2}\vert\right]
\end{align}
where the process $Z$  {solves}  the ${\cal P}_t$ {- SDE} :
\begin{align}
dZ_t&=b(t,Z_t)dt+h^{1/2}dW_t.
\end{align}
and 
\begin{align}\label{backward H}
\tilde H^*(t,x)=\inf_{b^*\in{\cal U}^*}E_{t,x}\left[\int_{-T/2}^{t} (b^*(s,Z_s))^2  ds+\log(\vert Z_{-T/2\vee\delta}\vert+1)\right]
\end{align}
where the process $Z$  {solves}  the ${\cal P}_t$ {- SDE} :
$$
d_*Z_t=b^*(t,Z_t)d_*t+h^{1/2}d_*W^*_t.
$$

The existence and uniqueness of solution for the free boundary problems below is shown for instance in Cannon \cite{cannon1984one}.

\subsection{The forward stochastic control problem}
\label{The forward stochastic control problem}

As seen in the previous sections, to find the value function $U$ associated with the forward control problem, one has to solve the free-boundary problem 
$$
\begin{cases}
-\frac{\partial U}{\partial t}(t,x)+\frac{1}{2}\left(\frac{\partial U}{\partial x}(t,x)\right)^2-\frac{\hbar }{2}\frac{\partial^2 U}{\partial x^2}(t,x)=0, &(t,x)\in{\cal C}\\
U(t,x)=\vert x\vert,&(t,x)\in{\cal S}\\
U(T/2,x)=\vert x\vert,& x\in\mathbb{R}
\end{cases}.
$$ 
Therefore, one of the first steps to solve the control problem is to guess the shape of the continuation and stopping regions. For this particular case, given the shape of the terminal cost $S(x)=\vert x\vert$ and the fact that $V(x)=0$,  it follows from $J_{t,x}(Z,\tau,b)\geq 0$, for all $\tau\in{\cal T}_t$ and  $b\in \cal U$,  that the forward stopping region ${\cal S}\supset\{0\}$ and $U(t,0)=0$ for all $t\in[-T/2,T/2]$. Additionally, one can check that, for $x\in\mathbb{R}\setminus \{0\}$,
$$
-\frac{\partial S}{\partial t}(x)+\frac{1}{2}\left(\frac{\partial S}{\partial x}(x)\right)^2-\frac{h}{2}\frac{\partial^2 S}{\partial x^2}(x)=\frac{1}{2}>0,
$$ 
meaning that $U(t,x)\neq S(x)$ for all $x\neq 0$. In other words, the value function $U$ can be described as follows: 
\begin{itemize}
	\item[(i)] for $x<0$, $U$ is the unique classical solution to the boundary problem
\begin{equation*}\label{BP1}
\begin{cases}
-\frac{\partial U}{\partial t}(t,x)+\frac{1}{2}\left(\frac{\partial U}{\partial x}(t,x)\right)^2-\frac{\hbar }{2}\frac{\partial^2 U}{\partial x^2}(t,x)=0,&(t,x)\in[-T/2,T/2)\times(-\infty,0)\\
U(t,0)=0,& t\in[-T/2,T/2)\\
U(T/2,x)=-x,& x\in(-\infty,0)
\end{cases}.
\end{equation*}
\item[(ii)] for $x>0$, $U$ the unique classical solution to the boundary problem 
\begin{equation*}\label{BP2}
\begin{cases}
-\frac{\partial U}{\partial t}(t,x)+\frac{1}{2}\left(\frac{\partial U}{\partial x}(t,x)\right)^2-\frac{\hbar }{2}\frac{\partial^2 U}{\partial x^2}(t,x)=0,& (t,x)\in[-T/2,T/2)\times(0,\infty)\\
U(t,0)=0,& t\in[-T/2,T/2)\\
U(T/2,x)=x,& x\in(0,\infty)
\end{cases}.
\end{equation*}
\end{itemize}
An analytic expression for $U$ can be found by using the change of variable presented at the end of Section \ref{A uniqueness result}. Indeed, if $U(t,x)=-\hbar \log(\eta(t,x))$, then,  for $x<0$, $\eta$ is unique solution of the boundary and final problem
\begin{equation*}\label{BPHE1}
\begin{cases}
\frac{\partial \eta}{\partial t}(t,x)+\frac{\hbar }{2}\frac{\partial^2 \eta}{\partial x^2}(t,x)=0,&(t,x)\in[-T/2,T/2)\times(-\infty,0)\\
\eta(t,0)=1,& t\in[-T/2,T/2)\\
\eta(T/2,x)=e^{\frac{1}{\hbar }x},& x\in(-\infty,0)
\end{cases}.
\end{equation*}
and, when $x>0$, $\eta$ is the unique solution of
\begin{equation*}\label{BPHE2<}
\begin{cases}
\frac{\partial \eta}{\partial t}(t,x)+\frac{\hbar }{2}\frac{\partial^2 \eta}{\partial x^2}(t,x)=0,& (t,x)\in[-T/2,T/2)\times(0,\infty)\\
\eta(t,0)=1,&  t\in[-T/2,T/2)\\
\eta(T/2,x)=e^{-\frac{1}{\hbar }x},& x\in(0,\infty)
\end{cases}.
\end{equation*}
It is a matter of calculations to see that, for $t\in[-T/2,T/2]$
\begin{equation}
\eta(t,x)=\begin{cases}1+\frac{1}{\sqrt{2\pi \hbar (T/2-t)}}\int_{-\infty}^0\left(e^{-\frac{(x-y)^2}{2\hbar (T/2-t)}}-e^{-\frac{(x+y)^2}{2\hbar (T/2-t)}}\right)\left(e^{\frac{y}{\hbar }}-1\right)dy, & x<0\\
1+\frac{1}{\sqrt{2\pi \hbar (T/2-t)}}\int_{0}^\infty\left(e^{-\frac{(x-y)^2}{2\hbar (T/2-t)}}-e^{-\frac{(x+y)^2}{2\hbar (T/2-t)}}\right)\left(e^{-\frac{y}{\hbar }}-1\right)dy, & x> 0
\end{cases}.
\end{equation}

In this case, the optimal strategy, $(\hat b,\hat \tau)$, is the following
\begin{align*}
\hat b(t,x)=\begin{cases}
\frac{e^{\frac{T/2-t}{2 h}}}{\eta(t,x)\sqrt{2 \pi h (T/2-t)}}\left({e^{\frac{x}{h}}} \int_{-\infty }^{-x} {e^{-\frac{\left(y-(T/2-t)\right)^2}{2h (T/2-t)}}} \, dy+{e^{-\frac{ x}{h}}} \int_{-\infty }^x {e^{- \frac{\left(y-(T/2-t)\right)^2}{2h (T/2-t)}}}\, dy\right),&x<0\\
\frac{-e^{\frac{T/2-t}{2 h}}}{\eta(t,x)\sqrt{2 \pi  h (T/2-t)}} \left(e^{\frac x h} \int_{-\infty }^{-x} e^{- \frac{\left(y-(T/2-t)\right)^2}{2h (T/2-t)}} \, dy+e^{-\frac{x}{h}} \int_{-\infty }^x e^{- \frac{\left(y-(T/2-t)\right)^2}{2h (T/2-t)}} \, dy\right)&x>0
\end{cases}
\end{align*}
and
$$
\hat\tau=\inf\{u\geq t:\, Z_u =0\}. 
$$
Since the process is optimally stopped once it reaches the level $0$, and the terminal condition is the absolute value of the current state of the process, we are, indeed, constructing two ``symmetric'' versions of the same stochastic process: one when the initial condition is negative and a second one when the initial condition is positive.

\subsection{The backward stochastic control problem}
\label{The backward stochastic control problem}
The value function $U^*$ can be obtained as a solution to the free-boundary problem
$$
\begin{cases}
\frac{\partial U^*}{\partial t}(t,x)+\frac{1}{2}\left(\frac{\partial U^*}{\partial x}(t,x)\right)^2-\frac{\hbar }{2}\frac{\partial^2 U^*}{\partial x^2}(t,x)=0, &(t,x)\in{\cal C}^*\\
U^*(t,x)=\log(\vert x\vert+1),&(t,x)\in{\cal S}^*\\
U^*(-T/2,x)=\log(\vert x\vert+1),& x\in\mathbb{R}
\end{cases}.
$$
A similar argument to the  one used in the previous case allows us to get that ${\cal S}^*\supset\{0\}$ and $U^*(t,0)=0$ for all $t\in[-T/2,T/2]$. Moreover, it is a matter of calculations to see that, when $x\in\mathbb{R}\setminus\{0\}$,
$$
\frac{\partial S^*}{\partial t}(x)+\frac{1}{2}\left(\frac{\partial S^*}{\partial x}(x)\right)^2-\frac{\hbar }{2}\frac{\partial^2 S^*}{\partial x^2}(x)=\frac{\hbar +1}{2(1+\vert x\vert )^2}>0.
$$
Therefore, the backward stopping region is ${\cal S}^*=\{0\}$, which means that the following statements are true:
\begin{itemize}
\item[(i)] for $x<0$, $U^*$ is the unique solution to the boundary problem 
\begin{equation*}\label{FBP1}
\begin{cases}
\frac{\partial U^*}{\partial t}(t,x)+\frac{1}{2}\left(\frac{\partial U^*}{\partial x}(t,x)\right)^2-\frac{\hbar }{2}\frac{\partial^2 U^*}{\partial x^2}(t,x)=0,&(t,x)\in(-T/2,T/2]\times(-\infty,0)\\
U^*(t,0)=0,& t\in(-T/2,T/2]\\
U^*(-T/2,x)=\log(-x+1),& x\in(-\infty,0)
\end{cases},
\end{equation*}
\item[(ii)] when $x>0$, $U^*$ is the unique solution to the boundary problem
\begin{equation*}\label{FBP2}
\begin{cases}
\frac{\partial U^*}{\partial t}(t,x)+\frac{1}{2}\left(\frac{\partial U^*}{\partial x}(t,x)\right)^2-\frac{\hbar }{2}\frac{\partial^2 U^*}{\partial x^2}(t,x)=0,& (t,x)\in(-T/2,T/2]\times(0,\infty)\\
U(t,0)=0,& t\in(-T/2,T/2]\\
U(-T/2,x)=\log(x+1),& x\in(0,\infty)
\end{cases}.
\end{equation*}
\end{itemize}
Using the transformation $\eta^*(t,x)=e^{-\frac{1}{h}U^*(t,x)}$, one can obtain two equivalent boundary problems. If $x<0$, then $\eta$ is the unique solution to
\begin{equation}\label{FBPHE1}
\begin{cases}
\frac{\partial \eta^*}{\partial t}(t,x)-\frac{\hbar }{2}\frac{\partial^2 \eta^*}{\partial x^2}(t,x)=0,&(t,x)\in(-T/2,T/2]\times(-\infty,0)\\
\eta^*(t,0)=1,& t\in(-T/2,T/2]\\
\eta^*(-T/2,x)=(1-x)^{-\frac{1}{\hbar }},& x\in(-\infty,0)
\end{cases},
\end{equation}
and, when $x>0$, $\eta$ is the unique solution to
\begin{equation}\label{FBPHE2<}
\begin{cases}
\frac{\partial \eta^*}{\partial t}(t,x)-\frac{\hbar }{2}\frac{\partial^2 \eta^*}{\partial x^2}(t,x)=0,& (t,x)\in(-T/2,T/2]\times(0,\infty)\\
\eta^*(t,0)=1,&  t\in(-T/2,T/2]\\
\eta^*(-T/2,x)=(1+x)^{-\frac{1}{\hbar }},& x\in(0,\infty)
\end{cases}.
\end{equation}
Therefore, one  obtains
\begin{equation}
\eta^*(t,x)=\begin{cases}1+\frac{1}{\sqrt{2\pi \hbar (t+T/2)}}\int_{-\infty}^0\left(e^{-\frac{(x-y)^2}{2\hbar (t+T/2)}}-e^{-\frac{(x+y)^2}{2\hbar (t+T/2)}}\right)\left((1-y)^{-\frac{1}{\hbar }}-1\right)dy, & x<0\\
1+\frac{1}{\sqrt{2\pi \hbar (t+T/2)}}\int_{0}^\infty\left(e^{-\frac{(x-y)^2}{2\hbar (t+T/2)}}-e^{-\frac{(x+y)^2}{2\hbar (t+T/2)}}\right)\left((y+1)^{-\frac{1}{\hbar }}-1\right)dy, & x> 0
\end{cases}.
\end{equation}

The optimal strategy for the backward control problem is given by
\begin{align*}
	\hat{b}^*(t,x)&=\begin{cases}\frac{\int_{-\infty }^0  \left({(x-y) e^{-\frac{(x-y)^2}{2 \hbar (t+T/2)}}}-{(x+y) e^{-\frac{(x+y)^2}{2 \hbar (t+T/2)}}}\right) \left((1-y)^{-1/\hbar }-1\right)\, dy}{\eta^*(t,x)(t+T/2)^\frac{3}{2}\sqrt{2\pi \hbar}},& x<0\\
	\frac{\int_0^{\infty } \left({(x-y) e^{-\frac{(x-y)^2}{2 \hbar (t+T)}}-(x+y) e^{-\frac{(x+y)^2}{2 \hbar (t+T)}}}\right) \left((y+1)^{-1/\hbar }-1\right)\, dy}{\eta^*(t,x)(t+T/2)^\frac{3}{2}\sqrt{2\pi \hbar}},& x>0
	\end{cases}\\
	\hat{\tau}^*&=\sup\{-T/2\leq s\leq t:\,Z_s=0\}.
\end{align*}
Given the structure of the terminal cost, the process is stopped once it attains the level zero. Therefore, the process is well defined in the space-time domain $\mathbb{R}\setminus\{0\}
\times [\hat{\tau}^*,\hat{\tau}]$.

\subsection{Classical control problems}
 In this section, we will construct a Bernstein stochastic process by solving the control problems \eqref{forward H} and \eqref{backward H} and we will compare it with the optimal process $Z$ constructed above.

To solve the optimal control problems \eqref{forward H} and \eqref{backward H}, one may use the standard theory (see for instance Fleming and Soner \cite{fleming2006controlled})). This means that the value function of stochastic control problems  \eqref{forward H} and \eqref{forward H} are the unique classical solutions of the final and initial  boundary problems
\begin{equation}\label{BP_Classical_forward}
\begin{cases}
-\frac{\partial \tilde{H}}{\partial t}(t,x)+\frac{1}{2}\left(\frac{\partial \tilde{H}}{\partial x}(t,x)\right)^2-\frac{\hbar }{2}\frac{\partial^2 \tilde{H}}{\partial x^2}(t,x)=0,&(t,x)\in[-T/2,T/2)\times\mathbb{R}\\
\tilde{H}(T/2,x)=\vert x\vert ,& x\in\mathbb{R}
\end{cases}.
\end{equation}
and
\begin{equation}\label{BP_Classical_backward}
\begin{cases}
\frac{\partial \tilde{H}^*}{\partial t}(t,x)+\frac{1}{2}\left(\frac{\partial \tilde{H}^*}{\partial x}(t,x)\right)^2-\frac{\hbar }{2}\frac{\partial^2 \tilde{H}^*}{\partial x^2}(t,x)=0,&(t,x)\in(-T/2,T/2]\times\mathbb{R}\\
\tilde{H}^*(-T/2,x)=\log(-x+1),& x\in\mathbb{R}
\end{cases},
\end{equation}

Noticing that there is a unique positive solution of the adjoint boundary problems
\begin{equation}
\begin{cases}
\frac{\partial \eta}{\partial t}(t,x)+\frac{\hbar }{2}\frac{\partial^2 \eta}{\partial x^2}(t,x)=0,&(t,x)\in[-T/2,T/2)\times\mathbb{R}\\
\eta(T/2,x)=e^{-\frac{1}{\hbar }\vert x\vert },& x\in\mathbb{R}
\end{cases},
\end{equation}
and
\begin{equation}
\begin{cases}
\frac{\partial \eta^*}{\partial t}(t,x)-\frac{\hbar }{2}\frac{\partial^2 \eta^*}{\partial x^2}(t,x)=0,&(t,x)\in(-T/2,T/2]\times\mathbb{R}\\
\eta^*(-T/2,x)=(1+\vert x\vert)^{-\frac{1}{\hbar }},& x\in\mathbb{R}
\end{cases},
\end{equation}
namely,
\begin{align}
\eta(t,x)&=\frac{1}{\sqrt{2\pi \hbar (T/2-t)}}\int_{-\infty}^\infty e^{-\frac{(x-y)^2}{2\hbar (T/2-t)}}e^{-\frac{\vert y\vert}{\hbar }}dy,\\
\eta^*(t,x)&=\frac{1}{\sqrt{2\pi \hbar (t+T/2)}}\int_{-\infty}^\infty e^{-\frac{(x-y)^2}{2\hbar (t+T/2)}}(1+\vert y\vert)^{-\frac{1}{\hbar}}dy
\end{align}
the solution to \eqref{BP_Classical_forward} and \eqref{BP_Classical_backward}, can be found by using the change of variable $\tilde{H}(t,x)=-\hbar \log(\eta(t,x))$ and $\tilde{H}^*(t,x)=-\hbar \log(\eta^*(t,x))$.

The forward and backward optimal strategies are, in this case, given by the control functions 
\begin{align*}
\hat b(t,x)&= 2e^{\frac{T/2 -t-2 x}{2 \hbar }}\frac{ \int_{-\infty }^x e^{ -\frac{1}{2} \frac{(y-(T/2-t))^2}{{\hbar (T/2-t)}}} \, dy}{\eta(t,x) \sqrt{2 \pi  \hbar (T/2-t)}  }-1\\
\hat b^*(t,x)&=-\frac{\int_{-\infty }^{\infty}  {(x-y) e^{-\frac{1}{2}\frac{(x-y)^2}{\hbar (t+T/2)}}} \left((1-y)^{-1/\hbar }-1\right)\, dy}{\eta^*(t,x)(t+T/2)^\frac{3}{2}\sqrt{2\pi \hbar }},
\end{align*}
for every $(t,x)\in(-T/2,T/2)\times\mathbb{R}\setminus\{0\}$. They are different from the ones obtained in Sections \ref{The forward stochastic control problem} and \ref{The backward stochastic control problem}. 
 
By comparing the process $Z$ constructed at Sections \ref{The forward stochastic control problem} and \ref{The backward stochastic control problem} with the one obtained in the present section, we  conclude that the presence of random times in the action functionals changes effectively the optimal stochastic process. Indeed, the Bernstein process constructed in a random interval of time is different from  the one constructed in a deterministic interval of time, although they maximize the same action functionals.

\section{Acknowledgments}
This research was partly funded by FCT (Funda{\c c}\~ao para a Ci\^encia e Tecnologia, Portugal), through the project UID/MAT/00208/2019 and grant ``Schr\"{o}dinger's problem and Optimal Transport: a multidisciplinary perspective'', with reference PTDC/MAT-STA/28812/2017. 

Carlos Oliveira was partially supported by the Project CEMAPRE/REM - UIDB/05069/2020 - financed by FCT/MCTES through national funds.

\bibliographystyle{plain}
\bibliography{myrefs}
\end{document}